\documentclass{amsart}
\usepackage{amsmath}
\usepackage{amsfonts}
\usepackage{amssymb}
\usepackage{mathrsfs}
\usepackage{stmaryrd}
\usepackage{bbm}
\usepackage{oldgerm}
\usepackage[english]{babel}
\usepackage[T1]{fontenc}
\usepackage[latin1]{inputenc}
\usepackage[all]{xy}
\usepackage{hyperref}
\usepackage{tikz-cd}
\usepackage{extarrows}
\usepackage{enumerate}

\newtheorem{theorem}[subsection]{Theorem}
\newtheorem{proposition}[subsection]{Proposition}

\newtheorem{lemma}[subsection]{Lemma}

\theoremstyle{definition}

\newtheorem{remark}[subsection]{Remark}
\newtheorem{question}[subsection]{Question}

\numberwithin{equation}{subsection}

\usepackage{amsthm}

\begin{document}

\title {The exotic inverted Kloosterman sum}

\author{Lei Fu and Daqing Wan}
\address{Yau Mathematical Sciences Center, Tsinghua University, Beijing 100084, P. R. China} 
\email{leifu@mail.tsinghua.edu.cn}
\address{Department of Mathematics, University of California, Irvine, CA 92697, USA}
\email{dwan@math.uci.edu}


\date{}
\maketitle

\begin{abstract}
Let $B$ be a product of finitely many finite fields containing $\mathbb F_q$, $\psi:\mathbb F_q\to \overline{\mathbb Q}_\ell^*$ a
nontrivial additive character, and $\chi: B^*\to \overline{\mathbb Q}_\ell^*$ a multiplicative character.  
Katz introduced the so-called exotic
inverted Kloosterman sum 
\begin{eqnarray*}
\mathrm{EIK}(\mathbb F_q, a):=\sum_{\substack{x\in B^* \\ \mathrm{Tr}_{B/\mathbb F_q}(x)\not =0\\
\mathrm{N}_{B/\mathbb F_q}(x)=a}}
\chi(x)\psi\Big(\frac{1}{\mathrm{Tr}_{B/\mathbb F_q}(x)}\Big), \ \ a\in \mathbb F_q^*.
\end{eqnarray*}
We estimate this sum using $\ell$-adic cohomology theory. Our main result is that, up to a trivial term,  the associated exotic 
inverted Kloosterman sheaf is lisse of rank at most $2(n+1)$ and mixed of weight at most $n$, where $n+1 = \dim_{\mathbb F_q}B$. Up to a trivial main term, this gives the expected square root cancellation.

\end{abstract} 

\section{Introduction}

\subsection{The exotic inverted Kloosterman sum} Let $\mathbb F_q$ be a finite field of characteristic $p$ with $q$ elements, $B$ a finite {\'e}tale 
$\mathbb F_q$-algebra of degree $n+1\geq 2$, $\ell$ a prime number different from $p$,  $\chi: B^*\to \overline{\mathbb Q}_\ell^*$ a multiplicative character, 
and $\psi:\mathbb F_q\to \overline{\mathbb Q}_\ell^*$ a nontrivial additive character. Deligne \cite[Sommes trig. 7.18]{SGA4.5} introduces 
the following exotic Kloosterman sum
\begin{eqnarray*}
\mathrm{EK}(\mathbb F_q, a):=\sum_{\substack{x\in B^*, ~
\mathrm{N}_{B/\mathbb F_q}(x)=a}}
\chi(x)\psi\Big(\mathrm{Tr}_{B/\mathbb F_q}(x)\Big),  \ a\in \mathbb F_q^*.
\end{eqnarray*}
where ${\rm Tr}_{B/\mathbb F_q}$ (resp. ${\rm N}_{B/\mathbb F_q}$) 
denotes the trace (resp. norm) map from $B$ to $\mathbb F_q$.  
This sum is studied by Katz \cite[8.8.5]{K0} via his theory of 
exotic Kloosterman sheaf on the torus $\mathbb G_m$,  which is a lisse sheaf of rank $n+1$, pure of weight $n$. This 
implies the following square root estimate 
$$\left|\mathrm{EK}(\mathbb F_q, a) \right| \leq (n+1)q^{n/2}.$$
Note that in the split case where $B=\mathbb F_q^{n+1}$ and $\chi=1$, this estimate reduces to Deligne's well known square root estimate 
for the classical hyper-Kloosterman sum 
$$\Big|\sum_{\substack{x_i\in \mathbb F_q^*, \  x_1\cdots x_{n+1}=a}}
\psi(x_1+\cdots +x_{n+1})\Big| \leq (n+1)q^{n/2}. 
$$
By inverting $\mathrm{Tr}_{B/\mathbb F_q}(x)$ in the definition of 
$\mathrm{EK}(\mathbb F_q, a)$, Katz \cite{K} introduces the following more complicated 
 \emph{exotic inverted Kloosterman sum} 
\begin{eqnarray*}
\mathrm{EIK}(\mathbb F_q, a):=\sum_{\substack{x\in B^*\\ \mathrm{Tr}_{B/\mathbb F_q}(x)\not =0\\
\mathrm{N}_{B/\mathbb F_q}(x)=a}}
\chi(x)\psi\Big(\frac{1}{\mathrm{Tr}_{B/\mathbb F_q}(x)}\Big), \ \ a\in \mathbb F_q^*.
\end{eqnarray*}
Motivated by applications in Ramanujuan graphs, Katz \cite{K} himself studies the case $n=1$. 
Recently Lin-Wan \cite{LW} studies the split case $B=\mathbb F_q^{n+1}$ for general $n$ by reducing the inverted Kloosterman sum
to a toric exponential sum of a Laurent polynomial. 
This requires the condition $(n+1, p)=1$ in order for the Laurent polynomial to be non-degenerate. 
The case $(n+1, p)>1$ is left open in \cite{LW}.  In the present paper, we study the exotic inverted Kloosterman sum $\mathrm{EIK}(\mathbb F_q, a)$ in the general non-split case for all $n$. The key is the construction of 
the exotic inverted Kloosterman sheaf over the torus $\mathbb G_m$, which is shown to be lisse of rank at most $2(n+1)$ 
and mixed of weights at most $n$ if either $p>2$ or $p=2$ but $n \not\equiv 1 \mod 4$. In the exceptional case that 
$p=2$ and $n \equiv 1 \mod 4$, this exotic inverted sheaf is only lisse over the punctured torus 
$\mathbb G_m -\{ (\frac{n+1}{2})^{-(n+1)}\}$.   As a consequence of this cohomological result, we 
also obtain a square root estimate for the exotic inverted Kloosterman sum $\mathrm{EIK}(\mathbb F_q, a)$. 

To state our bound more precisely, let ${\bf N}^{-1}(1)= \{x\in B^*: \mathrm N_{B/\mathbb F_q}(x)=1\}$  
be the kernel of the norm map, which is a subgroup of $B^*$ with index $q-1$. 
If $\chi|_{{\bf N}^{-1}(1)}=1$, then  $\chi({{\bf N}^{-1}(a))}$ is well defined for all $a\in \mathbb F_q^*$. A weaker version of our main result is 

\begin{theorem}\label{thm0} Let $B$ be a finite {\'e}tale 
$\mathbb F_q$-algebra of degree $n+1$ with $(n+1, p)=1$.  If $\chi|_{{\bf N}^{-1}(1)}\not=1$, then for any $a \in \mathbb F_q^*$, we have 
$$\left|\mathrm{EIK}(\mathbb F_q, a) \right| \leq (2n+2)q^{n/2}.$$
If $\chi|_{{\bf N}^{-1}(1)}=1$, then we have
$$\left|\mathrm{EIK}(\mathbb F_q, a)  + \chi({{\bf N}^{-1}(a))} \frac{|B^*|}{q(q-1)} \right| \leq (2n+2)q^{n/2}.$$
\end{theorem}

As mentioned above, this bound is proved in \cite{LW} in the split case $B=\mathbb F_q^{n+1}$.  If $(n+1, p)>1$, the 
statement is more subtle and the estimate is somewhat stronger.  

\begin{theorem}\label{thm00} Let $B$ be a finite {\'e}tale 
$\mathbb F_q$-algebra of degree $n+1$ with $p|(n+1)$.  Write $n+1 =p^km \geq 2m$ with $(m, p)=1$. 
Suppose either $p^k>2$ or $p^k=2$ but $a\not =m^{-(n+1)}$. Suppose furthermore that 
$\psi=\psi_0\circ\mathrm{Tr}_{\mathbb F_q/\mathbb F_p}$
for some non-trivial additive character $\psi_0$ of $\mathbb F_p$.
If $\chi|_{{\bf N}^{-1}(1)}\not=1$, we have 
$$\left|\mathrm{EIK}(\mathbb F_q, a) \right| \leq (n+1)q^{n/2}.$$
If $\chi|_{{\bf N}^{-1}(1)}=1$, then we have
$$\left|\mathrm{EIK}(\mathbb F_q, a)  + \chi({{\bf N}^{-1}(a))} \frac{|B^*|}{q(q-1)} \right| \leq (n+1)q^{n/2}.$$

\end{theorem}

\begin{remark} (i) In the case $n=1$, Theorem \ref{thm0} and Theorem \ref{thm00}  reduce to Theorem 1 and Theorem 2 in Katz \cite{K}.

(ii) In the exceptional case $p=2$, $n \equiv 1 \mod 4$ and $a=m^{-(n+1)}=(\frac{n+1}{2})^{-(n+1)}$, the situation becomes
degenerate and our toric method does not apply. For example, if $p=2$, $n=1$ and $B=\mathbb F_q \times \mathbb F_q$, then the condition  $a=m^{-(n+1)}$ becomes $a=1$. In this case, as noted by Katz in \cite{K}, for $x \in \mathbb F_q^*$, we have 
$$ \frac{1}{x +\frac{1}{x}} = \frac{x}{x^2+1} = \frac{1}{x+1} - \frac{1}{(x+1)^2},$$
which is of the form $f^p-f$. Thus the exotic inverted Kloosterman sum $\mathrm{EIK}(\mathbb F_q, a)$ for $a=1$ 
is given by 
$$\mathrm{EIK}(\mathbb F_q, 1) = \sum_{x \in \mathbb F_q -\{0, 1\}} (\chi_1\chi_2^{-1}) (x),$$
which is trivial, where the two characters $\chi_i$ are the restriction of $\chi$ to the two factors of $B^* = \mathbb F_q^* \times \mathbb F_q^*$. 
For this reason, the exceptional case $a=1$ is also removed from the 
theorem in \cite{K}.   

(iii) The assumption $\psi=\psi_0\circ\mathrm{Tr}_{\mathbb F_q/\mathbb F_p}$ is for normalization only. 
Each non-trivial additive character $\psi'$ of $\mathbb F_q$ is of the form $\psi'(x) =\psi (bx)$ for a unique $b \in \mathbb F_q^*$. 
The exotic inverted Kloosterman sum associated to the new additive character $\psi'$ is then given by 
\begin{eqnarray*}
\sum_{\substack{x\in B^*\\ \mathrm{Tr}_{B/\mathbb F_q}(x)\not =0\\
\mathrm{N}_{B/\mathbb F_q}(x)=a}}
\chi(x)\psi\Big(\frac{b}{\mathrm{Tr}_{B/\mathbb F_q}(x)}\Big) = \chi(b)\cdot \mathrm{EIK}(\mathbb F_q, ab^{-(n+1)}), \ \ a\in \mathbb F_q^*.
\end{eqnarray*}
This is obtained by the change of variable $x \rightarrow bx$. Thus, without loss of generality, we can always assume that 
$\psi=\psi_0\circ\mathrm{Tr}_{\mathbb F_q/\mathbb F_p}$.  
\end{remark}

To state the full cohomological version of our result and the construction of the exotic inverted Kloosterman 
sheaf, we need to recall some background information.  

\subsection{Restriction of scalars of the additive and the multiplicative group schemes}
Let $\mathbb G_{a, B}=\mathrm{Spec}\, B[t]$ (resp. $\mathbb G_{m, B}=\mathrm{Spec}\, B[t, 1/t]$) be the additive 
(resp. multiplicative) group scheme over $B$, and let $\mathrm{Res}_{B/\mathbb F_q}(\mathbb G_{a, B})$ (resp. 
$\mathrm{Res}_{B/\mathbb F_q}(\mathbb G_{m, B})$) be its Weil restriction of scalars. It is the commutative group
scheme over $\mathbb F_q$ representing the functor 
$$A\mapsto B\otimes_{\mathbb F_q} A \quad (\hbox{resp. } A\mapsto (B\otimes_{\mathbb F_q} A)^*)$$
from the category of $\mathbb F_q$-algebras to the category of abelian groups. 
Consider the map
$$\mathrm{Tr}_{B\otimes_{\mathbb F_q} A/A}: B\otimes_{\mathbb F_q} A\to A
\quad (\hbox{resp. }\mathrm N_{B\otimes_{\mathbb F_q} A/A}: (B\otimes_{\mathbb F_q} A)^*\to A^*)$$ 
sending every $x\in B\otimes_{\mathbb F_q} A$ to the trace (resp. determinant) of the $A$-linear map 
on the free $A$-module $B\otimes_{\mathbb F_q} A$ defined by the multiplications by $x$. 
It is functorial in $A$, and defines a homomorphism of group schemes
$$\mathbf{Tr}: \mathrm{Res}_{B/\mathbb F_q}(\mathbb G_{a, B})\to \mathbb G_{a, \mathbb F_q}\quad 
(\hbox{resp. } \mathbf{N}: \mathrm{Res}_{B/\mathbb F_q}(\mathbb G_{m, B})\to \mathbb G_{m, \mathbb F_q}).$$ 

\subsection{The Lang sheaf} Let $G$ be a commutative algebraic group over $\mathbb F_q$. The \emph{Lang isogeny} is the 
the \'etale finite homomorphism of algebraic groups
$$L: G\to G, \quad x\mapsto Fx \cdot x^{-1},$$
where $F:G\to G$ is the Frobenius correspondence. The kernel of $L$ is the group $G(\mathbb F_q)$ of $\mathbb F_q$-points 
on $G$. The Lang isogeny is a $G(\mathbb F_q)$-torsor over $G$, which we call the \emph{Lang torsor}.
Let $\rho: G(\mathbb F_q)\to \overline{\mathbb Q}_\ell^*$ be a character. Pushing forward the Lang torsor by $\rho^{-1}$, we get a
lisse $\overline{\mathbb Q}_\ell$-sheaf $\mathcal L_\rho$ on $G$, which we call the \emph{Lang sheaf}. We have a homomorphism 
of groups 
$$\mathrm{N}_{\mathbb F_{q^m}/\mathbb F_q}: G(\mathbb F_{q^m})\to G(\mathbb F_q),\quad x\mapsto x \cdot F(x)\cdots F^{m-1}(x).$$
For any $x\in G(\mathbb F_{q^m})$, let $F_x\in\mathrm{Gal}(\overline{\mathbb F}_{q}/\mathbb F_{q^m})$ be the geometric Frobenius.
We have
$$\mathrm{Tr}(F_x, \mathcal L_{\rho, \bar x})=\rho(\mathrm{N}_{\mathbb F_{q^m}/\mathbb F_q}(x)).$$ 
Moreover, we have an isomorphism $$\mathcal L_\rho|_{G\otimes_{\mathbb F_q}\mathbb F_{q^m}}
\cong \mathcal L_{\rho\circ \mathrm{N}_{\mathbb F_{q^m}/\mathbb F_q}},$$
where  $\mathcal L_{\rho\circ \mathrm{N}_{\mathbb F_{q^m}/\mathbb F_q}}$ is the Lang sheaf on $G\otimes_{\mathbb F_q}\mathbb F_{q^m}$
associated to the character $\rho\circ \mathrm{N}_{\mathbb F_{q^m}/\mathbb F_q}: G(\mathbb F_{q^m})\to \overline{\mathbb Q}_\ell^*$. 
For details, confer \cite[Sommes trig. 1.4-1.8]{SGA4.5}.

The Lang sheaf $\mathcal L_\psi$ on $\mathbb G_{a, \mathbb F_q}$ associated to the character $\psi$ is called the 
\emph{Artin-Schreier sheaf}. Denote by $\mathcal K_\chi$ the Lang sheaf on 
$\mathrm{Res}_{B/\mathbb F_q}(\mathbb G_{m, B})$ associated to the character 
$$\mathrm{Res}_{B/\mathbb F_q}(\mathbb G_{m, B})(\mathbb F_q)\cong B^*\stackrel \chi\to \overline{\mathbb Q}_\ell^*.$$  
For any $x\in \mathrm{Res}_{B/\mathbb F_q}(\mathbb G_{m, B})(\mathbb F_{q^m})\cong (B\otimes_{\mathbb F_q}\mathbb F_{q^m})^*$, 
we have
\begin{eqnarray*}
\mathrm{Tr}(F_x, \mathcal K_{\chi, \bar x})=\chi(\mathrm{N}_{B\otimes_{\mathbb F_q} \mathbb F_{q^m}/
B}(x)).
\end{eqnarray*}
Confer \ref{norm} for an explanation of this fact.

\subsection{The exotic inverted Kloosterman complex}
Given $a\in\mathbb F^*_{q^m}\cong 
\mathbb G_{m,\mathbb F_q}(\mathbb F_{q^m})$, we are interested in the exotic inverted Kloosterman sum
$$\mathrm{EIK}(\mathbb F_{q^m},a):=\sum_{\substack{x\in (B\otimes_{\mathbb F_q}\mathbb F_{q^m})^* \\
\mathrm{Tr}_{B\otimes_{\mathbb F_q}\mathbb F_{q^m}/\mathbb F_{q^m}}(x)\not=0\\
\mathrm N_{B\otimes_{\mathbb F_q}\mathbb F_{q^m}/\mathbb F_{q^m}}(x)=a}}
\chi(\mathrm N_{B\otimes_{\mathbb F_q} \mathbb F_{q^m}/
B}(x)) \psi\Big(\mathrm{Tr}_{\mathbb F_{q^m}/\mathbb F_q}\Big(\frac{1}{\mathrm{Tr}_{B\otimes_{\mathbb F_q} \mathbb F_{q^m}/
\mathbb F_{q^m}}(x)}\Big)\Big).$$
The map on $\mathbb F_{q^m}$-points
$$\mathbf{Tr}: \mathrm{Res}_{B/\mathbb F_q}(\mathbb G_{a, B})(\mathbb F_{q^m})\to \mathbb G_{a,\mathbb F_q}(\mathbb F_{q^m}),
\quad  \mathbf N:
\mathrm{Res}_{B/\mathbb F_q}(\mathbb G_{m, B})(\mathbb F_{q^m})\to \mathbb G_{m,\mathbb F_q}(\mathbb F_{q^m})$$ 
induced by the morphisms $\mathbf{Tr}: \mathrm{Res}_{B/\mathbb F_q}(\mathbb G_{a, B})\to \mathbb G_{a,\mathbb F_q}$ 
and $\mathbf N: \mathrm{Res}_{B/\mathbb F_q}(\mathbb G_{m, B})\to \mathbb G_{m,\mathbb F_q}$
can be identified with the maps
$$\mathrm{Tr}_{B\otimes_{\mathbb F_q}\mathbb F_{q^m}/\mathbb F_{q^m}}: 
B\otimes_{\mathbb F_q}\mathbb F_{q^m} \to \mathbb F_{q^m}, 
\quad \mathrm N_{B\otimes_{\mathbb F_q}\mathbb F_{q^m}/\mathbb F_{q^m}}: 
(B\otimes_{\mathbb F_q}\mathbb F_{q^m})^*\to \mathbb F_{q^m}^\ast.$$
Let $U=\mathrm{Res}_{B/\mathbb F_q}(\mathbb G_{m,B})-\mathbf{Tr}^{-1}(0)$. Then we have a morphism 
$$1/\mathbf{Tr}: U\to \mathbb G_{a, \mathbb F_q}.$$ 
Define the \emph{exotic inverted Kloosterman complex} to be the object
$$\mathrm{EIK}:=R\mathbf N|_{U,!} \Big(\mathcal K_\chi|_U\otimes (1/\mathbf{Tr})^*\mathcal L_\psi\Big)$$
in the triangulated category $D_c^b(\mathbb G_{m, \mathbb F_q}, \overline{\mathbb Q}_\ell)$
of complexes of $\overline{\mathbb Q}_\ell$-sheaves defined in \cite[1.1.3]{D}.
By the Grothendieck trace formula (\cite[Rapport 3.2]{SGA4.5}), we have the following.

\begin{proposition} Notation as above. For any $a\in \mathbb F^*_{q^m}\cong 
\mathbb G_{m,\mathbb F_q}(\mathbb F_{q^m})$, we have 
$$\mathrm{Tr}(F_a, \mathrm{EIK}_{\bar a})=\mathrm{EIK}(\mathbb F_{q^m},a).$$
\end{proposition}

The main result of this paper is the following.

\begin{theorem}\label{thm} Let $n+1=\mathrm{dim}_{\mathbb F_q} B\geq 2$. In the triangulated category $D_c^b(\mathbb G_{m, \mathbb F_q}, \overline{\mathbb Q}_\ell)$, we have a 
distinguished triangle 
$$\mathcal H[-n]\to \mathrm{EIK} \to R\mathbf N_!\mathcal K_\chi(1)[1]\to$$
and $R\mathbf N_!\mathcal K_\chi$ and $\mathcal H$ have the following properties:

\begin{enumerate}[(i)]
\item If the restriction of $\chi$ to 
${\bf N}^{-1}(1)$ is nontrivial, then $R\mathbf N_!\mathcal K_\chi=0$. 
If the restriction of $\chi$ to 
${\bf N}^{-1}(1)$ is trivial, then there exists 
a character $\chi_0: \mathbb F_q^*\to \overline{\mathbb Q}_\ell^*$ such that $\chi=\chi_0\circ \mathrm N_{B/\mathbb F_q}$, and 
we have $R\mathbf N_!\mathcal K_\chi=\mathcal K_{\chi_0}\otimes R\mathbf N_!\overline{\mathbb Q}_\ell$, 
where $\mathcal K_{\chi_0}$ is the Kummer sheaf (i.e. the Lang sheaf) on $\mathbb G_{m, \mathbb F_q}$ associated to the 
character $\chi_0$. 
\item If $(p, n+1)=1$, then $\mathcal H$ is a mixed lisse $\overline{\mathbb Q}_\ell$-sheaf on $\mathbb G_{m, \mathbb F_q}$ of rank $2n+2$ 
with weights $\leq n$.
\item Suppose $p|(n+1)$ and $\psi=\psi_0\circ\mathrm{Tr}_{\mathbb F_q/\mathbb F_p}$ for some non-trivial additive 
character $\psi_0$ of $\mathbb F_p$. Write $n+1=p^km\geq 2m$ such that $(p,m)=1$. 
\begin{enumerate}[(1)]
\item If $n+1>2m$, then $\mathcal H$ is a mixed lisse $\overline{\mathbb Q}_\ell$-sheaf  on 
$\mathbb G_{m, \mathbb F_q}$ of rank $n+1$ with weights $\leq n$.
\item If $n+1=2m$ (namely, $p=2$ and $n \equiv 1 \mod 4$), then $\mathcal H|_{\mathbb G_{m, \mathbb F_q}-\{m^{-(n+1)}\}}$ is a mixed lisse $\overline{\mathbb Q}_\ell$-sheaf  on 
$\mathbb G_{m, \mathbb F_q}-\{m^{-(n+1)}\}$ of rank $n+1$ with weights $\leq n$.
\end{enumerate}
\end{enumerate}
\end{theorem}

We may call $\mathcal H$ in Theorem \ref{thm} the \emph{exotic inverted Kloosterman sheaf}. 

\medskip
Theorems \ref{thm0} and \ref{thm00} are immediate consequence of Theorem \ref{thm} and the trace formula. 
The condition $\chi|_{{\bf N}^{-1}(1)}=1$ can be made completely explicit as follows. 
Recall that a finite  \'{e}tale algebra $B$ over $\mathbb F_q$ of degree $n+1$ can be written in the form 
\begin{align}\label{eq}
B = \mathbb F_{q^{n_1}} \times \cdots \times \mathbb F_{q^{n_k}}, \quad n_1+\cdots + n_k = n+1.
\end{align}
The factorization type $\{n_1, \cdots, n_k\}$ of $B$ is uniquely determined up to permutation. 
For each $1\leq i \leq k$, let $\chi_i$ denote the restriction of $\chi$ to the $i$-th component $\mathbb F^*_{q^{n_i}}$ in  
the factorization (\ref{eq}). For any $x=(x_1, \cdots, x_k) \in B^*$, we have 
$$\chi(x) = \chi_1(x_1)\cdots \chi_k(x_k), \quad \psi ({\rm Tr}(x)) =\psi( {\rm Tr}_{\mathbb F_{q^{n_i}}/\mathbb F_q}
(x_1)+\cdots +{\rm Tr}_{\mathbb F_{q^{n_i}}/\mathbb F_q}(x_k)).$$
The exotic inverted Kloosterman sum becomes 
\begin{align}\label{sum}
\mathrm{EIK}(\mathbb F_q, a)=\sum_{\substack{x_i\in \mathbb F_{q^{n_i}}^*\\ 
\sum_{i=1}^k\mathrm{Tr}_{\mathbb F_{q^{n_i}}/\mathbb F_q}(x_i)\not =0\\
\prod_{i=1}^k \mathrm{N}_{\mathbb F_{q^{n_i}}/\mathbb F_q}(x_i)=a}}
\chi_1(x_1)\cdots \chi_k(x_k)\psi\Big(\frac{1}{\sum_{i=1}^k\mathrm{Tr}_{\mathbb F_{q^{n_i}}/\mathbb F_q}(x_i)}\Big).
\end{align}

\begin{proposition} Notation as above. We have 
$\chi|_{{\bf N}^{-1}(1)}=1$ if and only if there exists a character $\chi_0$ of $\mathbb F^*_q$ such that
$\chi_i= \chi_0 \circ {\rm N}_{\mathbb F_{q^{n_i}}/\mathbb F_q}$ for all $1\leq i\leq k$.
\end{proposition}

\begin{proof}
The homomorphism 
$$\mathrm N_{B/\mathbb F_q}: B^*= \mathbb F_{q^{n_1}}^* \times \cdots \times \mathbb F_{q^{n_k}}^*\to\mathbb F^*_q ,\quad
(x_1, \ldots, x_n)\mapsto \prod_{i=1}^k \mathrm N_{\mathbb F_{q^{n_i}}/\mathbb F_q}(x_i)$$ is surjective with kernel 
$\mathbf N^{-1}(1)$. We have $\chi|_{{\bf N}^{-1}(1)}=1$ if and only if there exists 
a character $\chi_0$ of $\mathbb F^*_q$ such that
$\chi= \chi_0 \circ \mathrm N_{B/\mathbb F_q}$, that is, 
$$\prod_{i=1}^m \chi_i(x_i)=\prod_{i=1}^m \chi_0(\mathrm N_{\mathbb F_{q^{n_i}}/\mathbb F_q}(x_i)).$$
The last condition is equivalent to 
$\chi_i=\chi_0 \circ {\rm N}_{\mathbb F_{q^{n_i}}/\mathbb F_q}$ for all $\leq i \leq k$.
\end{proof}

To prove Theorem \ref{thm}, we first recall the properties of the restriction of scalars in section $2$. 
This will reduce the problem to the split case. Then we prove Theorem \ref{thm} in section $3$  by
a cohomological version of the toric reduction in \cite{LW} and applying 
the theory of $\ell$-adic GKZ hypergeometric sheaf developed in \cite{F}.  This works well in the case $(n+1, p)=1$ 
as the associated Laurent polynomial is non-degenerate. In the case $(n+1, p)>1$, the situation is more 
subtle as the Laurent polynomial is degenerate. A key new ingredient is a further reduction so that the 
final Laurent polynomial becomes non-degenerate 
and thus the theory of GKZ hypergeometric sheaf can still apply.

\subsection*{Acknowledgements}
LF is supported by 2021YFA1000700, 2023YFA1009703 and NSFC12171261.

\section{Restriction of scalars}

Let $S$ be a scheme. Denote by $(\mathbf {Sch}_S)$ the category of $S$-schemes. 
For any $S$-scheme $T$, let $$\eta_T=\mathrm{Hom}_S(-, T)$$ be the functor represented by $T$. 
Let $h:S'\to S$ be a morphism of schemes. For any functor $F':(\mathbf{Sch}_{S'})^\circ\to (\hbox{Sets})$, 
let $h_*F': (\mathbf{Sch}_S)^\circ\to (\hbox{Sets})$ be the 
functor $$(h_*F')(T)=F'(T\times_S S').$$ 
Note that for any $S'$-scheme $X'$,  $h_*\eta_{X'}$ is the functor $T\mapsto \mathrm{Hom}_{S'}(T\times_S S', X')$.

\begin{proposition} \label{res} Let $h:S'\to S$ be a finite flat morphism. Suppose $X'$ is an $S'$-scheme such 
that for any $S$-scheme $\mathrm{Spec}\, K$ with $K$ being a field, the image of any $S'$-morphism
$\mathrm{Spec}\,K\times_S S'\to X'$ is contained in an affine open subset of $X'$. 

(i) The functor
$$h_*\eta_{X'}:(\mathbf{Sch}_S)^\circ \to (\mathbf{Sets}),\quad T\mapsto \mathrm{Hom}_{S'}(T\times_S S', X')$$
is representable by an $S$-scheme $\mathrm{Res}_{S'/S}(X')$:
$$\mathrm{Hom}_{S'}(T\times_S S', X')\cong \mathrm{Hom}_{S}(T, \mathrm{Res}_{S'/S}(X')).$$

(ii) Let $T$ be an $S$-scheme and let $T'=T\times_S S'$. We have
$$\mathrm{Res}_{S'/S}(X')\times_S T\cong \mathrm{Res}_{T'/T}(X'\times_{S'}T').$$
\end{proposition}

We call the functor $\mathrm{Res}_{S'/S}:(\mathbf{Sch}_{S'})\to(\mathbf{Sch}_S)$ \emph{restriction of scalars}. 
It is right adjoint to the base change functor $T\mapsto T\times_S S'$.

\begin{proof} ${}$

(i) Confer \cite[7.6 Theorem 4]{BLM}.

(ii) For any $T$-scheme $Q$, let $Q'=Q\times_SS'$. We have 
\begin{align*}
\mathrm{Hom}_T(Q,  \mathrm{Res}_{T'/T}(X'\times_{S'}T'))\cong\,& \mathrm{Hom}_{T'}(Q', X'\times_{S'}T')\\
\cong\,& \mathrm{Hom}_{S'}(Q', X')\\
\cong\,& \mathrm{Hom}_{S}(Q, \mathrm{Res}_{S'/S}(X'))\\
\cong\,&\mathrm{Hom}_{T}(Q, \mathrm{Res}_{S'/S}(X')\times_S T).\qedhere
\end{align*}
\end{proof} 

If $G'$ is a group scheme over $S'$, then $h_*\eta_{G'}$ is a functor from $(\mathbf{Sch}_{S})$ to the category 
of groups. Thus $\mathrm{Res}_{S'/S}(G')$ is a group scheme over $S$.  

\subsection{}\label{split} Consider the case where $S=\mathrm{Spec}\,R$
and $S'=\mathrm{Spec}\,R'$ are affine such that $R'$ is a free $R$-module of finite rank. Denote $\mathrm{Res}_{S'/S}$ by
$\mathrm{Res}_{R'/R}$. Consider the restriction of scalars of the additive group scheme $\mathbb G_{a,R'}$ and the multiplicative group scheme
$\mathbb G_{m,R'}$ over $\mathrm{Spec}\,R'$. As schemes, we have an open immersion  
$$\mathrm{Res}_{R'/R}(\mathbb G_{m,R'})\hookrightarrow \mathrm{Res}_{R'/R}(\mathbb G_{a,R'}).$$ 
For any $R$-algebra $A$, and any $R$-scheme $X$, let
$X(A)=\mathrm{Hom}_R(\mathrm{Spec}A, X)$. We have 
$$(\mathrm{Res}_{R'/R}(\mathbb G_{a, R'}))(A)\cong A\otimes_R R', \quad (\mathrm{Res}_{R'/R}(\mathbb G_{m, R'}))(A)\cong 
(A\otimes_R R')^*.$$
We have homomorphisms of groups
$$\mathrm{Tr}_{A\otimes_R R'/A}: A\otimes_R R'\to A, \quad \mathrm N_{A\otimes_R R'/A}: (A\otimes_R R')^*\to A^*$$ 
defined by the trace and determinant of the endomorphism defined by the multiplication on the free $A$-module 
$A\otimes_R R'$, and they are functorial in $A$.
They define homomorphisms of group schemes 
$$\mathbf{Tr}: \mathrm{Res}_{R'/R}(\mathbb G_{a, R'})\to \mathbb G_{a, R},\quad 
\mathbf{N}: \mathrm{Res}_{R'/R}(\mathbb G_{m, R'})\to \mathbb G_{m, R}.$$ 

\subsection{}\label{free} In the split case where $R'=R^n$, we have 
$$(\mathrm{Res}_{R^n/R}(\mathbb G_{a, R^n}))(A)\cong A^n, \quad (\mathrm{Res}_{R^n/R}(\mathbb G_{a, R'}))(A)\cong 
A^{*,n}.$$
We thus have 
$$\mathrm{Res}_{R^n/R}(\mathbb G_{a, R^n})\cong\mathbb  G_{a, R}^n, \quad
\mathrm{Res}_{R^n/R}(\mathbb G_{m, R^n})\cong\mathbb  G_{m, R}^n.$$
The morphism 
$\mathbf{Tr}: \mathrm{Res}_{R^n/R}(\mathbb G_{a, R^n})\to \mathbb G_{a, R}$ (resp. 
$\mathbf{N}: \mathrm{Res}_{R^n/R}(\mathbb G_{m, R^n})\to \mathbb G_{m, R}$) is identified with 
\begin{eqnarray*}
&&\mathbf{Tr}: \mathbb G^n_{a, R}\to \mathbb G_{a, R},\quad (x_1, \ldots, x_n)\mapsto x_1+\cdots+ x_n\\
(\hbox{resp.} &&\mathbf N:\mathbb G^n_{m, R}\to \mathbb G_{m, R},\quad (x_1, \ldots, x_n)\mapsto x_1\cdots x_n).
\end{eqnarray*}

\subsection{}\label{field} Suppose $R=k$ is a field and $R'=k'$ is a finite separable extension of $k$. Choose $\xi\in k'$ such that 
$k'=k[\xi]$. Let $f(t)\in k[t]$ be the minimal polynomial of $\xi$ over $k$, $\tilde k$ an extension of $k'$ containing all 
the roots 
$\xi_1=\xi, \ldots, \xi_n$ of $f(t)$, and 
$\sigma_i: k'\to \tilde k$ the $k$-homomorphism mapping $\xi$ to $\xi_i$.
We have 
$$\tilde k\otimes_k k'\cong \tilde k\otimes_k k[t]/(f(t))\cong \tilde k[t]/(f(t))\cong \tilde k^n.$$
The canonical homomorphism
$$ \tilde k\to \tilde k\otimes_k k'  \quad (\hbox{resp. } k'\to \tilde k\otimes_k k')$$
is identified with the homomorphism 
\begin{eqnarray*}
&&\tilde k\to \tilde k^n,\quad x\mapsto (x, \ldots, x)\\
(\hbox{resp.} &&k'\to \tilde k^n,\quad x\mapsto (\sigma_1(x), \ldots, \sigma_n(x))).
\end{eqnarray*}
By Proposition \ref{res} (ii) and \ref{free}, we have
\begin{eqnarray*}
\mathrm{Res}_{k'/k}(\mathbb G_{a, k'})\otimes_k \tilde k\cong \mathbb G^n_{a, \tilde k},\quad
\mathrm{Res}_{k'/k}(\mathbb G_{m, k'})\otimes_k \tilde k\cong \mathbb G^n_{m, \tilde k}.
\end{eqnarray*}
and the base change to $\tilde k$ of the morphism 
$\mathbf{Tr}: \mathrm{Res}_{k'/k}(\mathbb G_{a, k'})\to \mathbb  G_{a, k}$ (resp. 
$\mathbf{N}: \mathrm{Res}_{k'/k}(\mathbb G_{m, k'})\to \mathbb G_{m, k}$) is identified with 
\begin{eqnarray*}
&&\mathbf{Tr}: \mathbb G^n_{a, \tilde k}\to \mathbb G_{a, \tilde k},\quad (x_1, \ldots, x_n)\mapsto x_1+\cdots+ x_n\\
(\hbox{resp.} &&\mathbf N:\mathbb G^n_{m, \tilde k}\to \mathbb G_{m, \tilde k},\quad (x_1, \ldots, x_n)\mapsto x_1\cdots x_n).
\end{eqnarray*}
The above results hold if $k'$ is a finite \'etale $k$-algebra and $\tilde k$ is a sufficiently large finite extension in 
the separable closure of $k$ so that $\tilde k\otimes_k k'$ splits into a product of finitely many copies of $\tilde k$.

\subsection{}\label{finitefield} Take $k=\mathbb F_q$. For any scheme $X$ over $\mathbb F_q$, 
let $F:X\to X$ be the Frobenius correspondence. It is the identity on $X$ as a map of set, and is the morphism  $s\mapsto s^q$ on the structure 
sheaf $\mathcal O_X$. For any $\mathbb F_{q^m}$-points $x:\mathrm{Spec}\,\mathbb F_{q^m}\to X$, we have a commutative diagram
$$\begin{array}{ccc}
\mathrm{Spec}\,\mathbb F_{q^m}&\stackrel{x} \to& X\\
{\scriptstyle F}\downarrow\quad&&\quad\downarrow {\scriptstyle F}\\
\mathrm{Spec}\,\mathbb F_{q^m}&\stackrel{x} \to& X,
\end{array}$$ where the Frobenius correspondence $F$ on $\mathrm{Spec}\,\mathbb F_{q^m}$ is induced by the Frobenius substitution
$$\phi: \mathbb F_{q^m}\to  \mathbb F_{q^m},\quad a\mapsto a^q.$$
So the $\mathbb F_{q^m}$-points $F(x)$ coincides with $x\circ F$. 

\subsection{}\label{norm} Let $B$ be a finite \'etale $\mathbb F_q$-algebra. We have 
$$\mathrm{Res}_{B/\mathbb F_q}(\mathbb G_{a, B})(\mathbb F_{q^m})=B\otimes_{\mathbb F_q} \mathbb F_{q^m},\quad
\mathrm{Res}_{B/\mathbb F_q}(\mathbb G_{m, B})(\mathbb F_{q^m})=(B\otimes_{\mathbb F_q} \mathbb F_{q^m})^*.$$
By \ref{finitefield}, the map on $\mathrm{Res}_{B/\mathbb F_q}(\mathbb G_{a, B})(\mathbb F_{q^m})$ (resp. 
$\mathrm{Res}_{B/\mathbb F_q}(\mathbb G_{m, B})(\mathbb F_{q^m})$) induced by 
the Frobenius correspondence can be identified with the map on $B\otimes_{\mathbb F_q} \mathbb F_{q^m}$
(resp. $(B\otimes_{\mathbb F_q} \mathbb F_{q^m})^*$) induced by $\mathrm{id}_B\otimes \phi$. For simplicity, we denote 
$\mathrm{id}_B\otimes\phi$ by $a\mapsto a^\phi$. The map
\begin{eqnarray*}
\mathrm{Tr}_{\mathbb F_{q^m}/\mathbb F_q}: \mathrm{Res}_{B/\mathbb F_q}(\mathbb G_{a, B})(\mathbb F_{q^m})\to 
\mathrm{Res}_{B/\mathbb F_q}(\mathbb G_{a, B})(\mathbb F_q),&\quad& x\mapsto x+Fx+\cdots +F^{m-1}x,\\
\mathrm N_{\mathbb F_{q^m}/\mathbb F_q}: \mathrm{Res}_{B/\mathbb F_q}(\mathbb G_{m, B})(\mathbb F_{q^m})\to 
\mathrm{Res}_{B/\mathbb F_q}(\mathbb G_{m, B})(\mathbb F_q),&\quad& x\mapsto x\cdot Fx\cdots \cdot F^{m-1}x
\end{eqnarray*}
can be identified with 
\begin{eqnarray*}
\mathrm{Tr}_{B\otimes_{\mathbb F_q} \mathbb F_{q^m}/
B}:B\otimes_{\mathbb F_q} \mathbb F_{q^m}\to 
B,&\quad& x\mapsto x+x^\phi+\cdots +x^{\phi^{m-1}},\\
\mathrm N_{B\otimes_{\mathbb F_q} \mathbb F_{q^m}/ 
B}: (B\otimes_{\mathbb F_q} \mathbb F_{q^m})^*\to 
B^*,&\quad& x\mapsto x\cdot x^\phi\cdots \cdot x^{\phi^{m-1}}
\end{eqnarray*}
Note that they are the trace and determinant of the endomorphism defined by the multiplication by $x$
on the free $B$-module $B\otimes_{\mathbb F_q} \mathbb F_{q^m}$. 

\begin{remark} Let's describe explicitly the morphism $N_{B\otimes_{\mathbb F_q}\mathbb F_{q^m}/B}$ in the 
case where $B=\mathbb F_{q^n}$ and $n|m$. In this case, $B\otimes_{\mathbb F_q}\mathbb F_{q^m}\cong \mathbb F_{q^m}^n$.
Fix an embedding $i: \mathbb F_{q^n}\hookrightarrow \mathbb F_{q^m}$. Then 
$i, i\phi, \ldots, i\phi^{n-1}$ are all the $\mathbb F_q$-embedding of $\mathbb F_{q^n}$ into $\mathbb F_{q^m}$. 
By \ref{field}, the Cartesian diagram 
$$\begin{tikzcd}
B\arrow[r]&B\otimes_{\mathbb F_q}\mathbb F_{q^m}\\
\mathbb F_q\arrow[u]\arrow [r]&\mathbb F_{q^m} \arrow[u]
\end{tikzcd}$$
can be identified with
$$\begin{tikzcd}
\mathbb F_{q^n}\arrow[r, "\Phi"]&\overbrace{\mathbb F_{q^m}\times\cdots \times \mathbb F_{q^m}}^{n\hbox{ copies}}\\
\mathbb F_q\arrow[u]\arrow [r]&\mathbb F_{q^m} \arrow[u,"\Psi"],
\end{tikzcd}$$
where 
$$\Phi(x)=(i(x), i(\phi^{-1}(x)), \ldots, i(\phi^{-(n-1)}(x))),\quad \Psi(x)=(x, \ldots, x).$$
The map $N_{B\otimes_{\mathbb F_q}\mathbb F_{q^m}/B}: B\otimes_{\mathbb F_q}\mathbb F_{q^m}\to B$ 
is identified with $N_{\Phi}: \mathbb F_{q^m}^n\to \mathbb F_{q^n}$, where 
\begin{eqnarray*}
N_\Phi(x_1, \ldots, x_n)&=&\prod_{j=1}^n N_\Phi(1, \ldots, 1, x_j, 1, \ldots, 1)\\
&=& \prod_{j=1}^n \phi^{j-1}(N_{\mathbb F_{q^m/\mathbb F_{q^n}}}(x_j))=\prod_{j=1}^nN_{\mathbb F_{q^m/\mathbb F_{q^n}}}(x_j)^{q^{j-1}}.
\end{eqnarray*}

\end{remark}

\section{Proof of the theorem}

For convenience, denote $\mathrm{Res}_{B/\mathbb F_q}$ by $\mathrm{Res}$. 
Let $f$ be the morphism 
$$f:\mathrm{Res}(\mathbb G_{m, B})\times \mathbb G_{m, \mathbb F_q} \times \mathbb G_{a, \mathbb F_q}\to \mathbb A^1, 
\quad (x, y, z)\mapsto y+z-y z\mathbf{Tr}(x),$$ 
and let $p_1$ and $p_2$ be the projections of $
\mathrm{Res}(\mathbb G_{m, B})\times \mathbb G_{m, \mathbb F_q} \times \mathbb G_{a, \mathbb F_q}$ to its first two factors. 

\begin{lemma}\label{lemma1} We have 
$$\mathrm{EIK}\cong R(\mathbf Np_1)_!(p_1^*\mathcal K_\chi \otimes f^*\mathcal L_\psi)(1)[2].$$ 
\end{lemma} 

\begin{proof} Let $g$ be the morphism 
$$g:\mathrm{Res}(\mathbb G_{m, B})\times \mathbb G_{m, \mathbb F_q} \times \mathbb G_{a, \mathbb F_q}\to \mathbb A^1, 
\quad (x, y, z)\mapsto z(1-y\mathbf{Tr}(x)),$$
let $$\pi: \mathrm{Res}(\mathbb G_{m, B})\times \mathbb G_{m, \mathbb F_q} \times \mathbb G_{a, \mathbb F_q}\to  
\mathrm{Res}(\mathbb G_{m, B})\times \mathbb G_{m, \mathbb F_q}, \quad (x, y, z)\mapsto (x, y)$$
be the projection, and let $q_1$ and $q_2$ be the projections  of $\mathrm{Res}(\mathbb G_{m, B})\times \mathbb G_{m, \mathbb F_q}$ 
to its factors. By the projection formula, we have 
\begin{eqnarray}\label{1}
R(\mathbf Np_1)_!(p_1^*\mathcal K_\chi \otimes f^*\mathcal L_\psi)
&\cong& R(\mathbf N q_1\pi)_!(\pi^*q_1^*\mathcal K_\chi \otimes \pi^* q_2^*\mathcal L_\psi\otimes g^*\mathcal L_\psi)\nonumber\\
&\cong& R(\mathbf Nq_1)_!(q_1^*\mathcal K_\chi \otimes q_2^*\mathcal L_\psi\otimes R\pi_!g^*\mathcal L_\psi).
\end{eqnarray}
Recall that $U=\mathrm{Res}(\mathbb G_{m,B})-\mathbf{Tr}^{-1}(0)$. 
We have a closed immersion $$k: U\to \mathrm{Res}(\mathbb G_{m, B})\times \mathbb G_{m, \mathbb F_q},\quad 
x\mapsto (x, 1/\mathbf{Tr}(x)).$$ 
Let $h$ be the morphism 
$$h: \mathrm{Res}(\mathbb G_{m, B})\times \mathbb G_{m, \mathbb F_q}\to \mathbb G_{a, \mathbb F_q}, \quad (x, y)\mapsto 1-y\mathbf{Tr}(x),$$ 
and let $\langle\, , \,\rangle$ be the pairing 
$$\langle\,,\, \rangle: \mathbb G_{a, \mathbb F_q} \times \mathbb G_{a, \mathbb F_q}\to  \mathbb G_{a, \mathbb F_q},\quad (w, w')\mapsto ww'.$$
We have a commutative diagram
$$
\begin{tikzcd}
 \mathrm{Res}(\mathbb G_{m, B})\times \mathbb G_{m, \mathbb F_q}\times \mathbb G_{a, \mathbb F_q} \arrow[rr, bend left,"g"]
 \arrow[d,"\pi"] \arrow[r,"h\times{\mathrm{id}}"]
  &\mathbb G_{a, \mathbb F_q}\times \mathbb G_{a, \mathbb F_q}  \arrow[d,"\mathrm{pr}_1"]
  \arrow[r,"\langle\, {,} \,\rangle"]& \mathbb G_{a, \mathbb F_q}\\
 \mathrm{Res}(\mathbb G_{m, B})\times \mathbb G_{m, \mathbb F_q}\arrow[r,"h"]&\mathbb G_{a, \mathbb F_q}& \\
 U\arrow[u,"k"]\arrow [r]&{0}\arrow[u]& 
 \end{tikzcd}
 $$ where $0$ is the zero section of $\mathbb G_{a, \mathbb F_q}$, 
 $\mathrm{pr}_1: \mathbb G_{a, \mathbb F_q}\times \mathbb G_{a, \mathbb F_q}\to 
 \mathbb G_{a, \mathbb F_q}$ is the projection to the first factor, and all squares are Cartesian.
 By \cite[1.2.2.2]{L}, we have
$$R\mathrm{pr}_{1!} \langle\,,\,\rangle^*\mathcal L_\psi\cong 0_*\overline{\mathbb Q}_\ell(-1)[-2].$$ 
By the proper base change theorem, we have 
\begin{eqnarray*}
R\pi_! g^*\mathcal L_\psi&\cong & R\pi_! (h\times\mathrm{id})^*\langle\,,\,\rangle^*\mathcal L_\psi
\cong h^* R\mathrm{pr}_{1!} \langle\,,\,\rangle^*\mathcal L_\psi\\&\cong& h^* 0_*\overline{\mathbb Q}_\ell(-1)[-2]
\cong k_*\overline{\mathbb Q}_\ell(-1)[-2].
\end{eqnarray*}
Combined with (\ref{1}), we get
\begin{align*}
R(\mathbf Np_1)_!(p_1^*\mathcal K_\chi \otimes f^*\mathcal L_\psi)
\cong\,& R(\mathbf Nq_1)_!(q_1^*\mathcal K_\chi \otimes q_2^*\mathcal L_\psi\otimes k_*\overline{\mathbb Q}_\ell(-1)[-2])\\
\cong\,& R(\mathbf Nq_1 k)_!(k^*q_1^*\mathcal K_\chi \otimes k^*q_2^*\mathcal L_\psi)(-1)[-2]\\
\cong\,& R\mathbf N|_{U,!} \Big(\mathcal K_\chi|_U\otimes (1/\mathbf{Tr})^*\mathcal L_\psi\Big)(-1)[-2]=\mathrm{EIK}(-1)[-2].\qedhere
\end{align*}
\end{proof}

Let $i$ be the closed immersion
$$i: \mathrm{Res}(\mathbb G_{m, B})\times \mathbb G_{m, \mathbb F_q}\to
\mathrm{Res}(\mathbb G_{m, B})\times \mathbb G_{m, \mathbb F_q} \times \mathbb G_{a, \mathbb F_q}, \quad (x, y)\mapsto (x, y,0),$$ and let 
$j$ be the complement open immersion
$$j: \mathrm{Res}(\mathbb G_{m, B})\times \mathbb G_{m, \mathbb F_q} \times \mathbb G_{m, \mathbb F_q}\hookrightarrow
\mathrm{Res}(\mathbb G_{m, B})\times \mathbb G_{m, \mathbb F_q} \times \mathbb G_{a, \mathbb F_q}.$$
We have  a distinguished triangle 
\begin{eqnarray}\label{triangle}
&&\\
&&R(\mathbf Np_1j)_!(j^*p_1^*\mathcal K_\chi \otimes j^*f^*\mathcal L_\psi)
\to R(\mathbf Np_1)_!(p_1^*\mathcal K_\chi \otimes f^*\mathcal L_\psi)
\to R(\mathbf Np_1i)_!(i^*p_1^*\mathcal K_\chi \otimes i^*f^*\mathcal L_\psi)\to\nonumber.
\end{eqnarray}

\begin{lemma}\label{lemma2} 
\begin{enumerate}[(i)]
\item We have 
$$R(\mathbf Np_1i)_!(i^*p_1^*\mathcal K_\chi \otimes i^*f^*\mathcal L_\psi)\cong R\mathbf N_! \mathcal K_\chi[-1].$$

\item If the restriction of $\chi: B^*\to \overline{\mathbb Q}_\ell^*$ to
${\bf N}^{-1}(1)=\{x\in B^*: \mathrm N_{B/\mathbb F_q}(x)=1\}$  is nontrivial, then $R\mathbf N_! \mathcal K_\chi=0$.

\item Suppose the restriction of $\chi$ to 
${\bf N}^{-1}(1)$ is trivial. Let $\chi_0: \mathbb F_q^*\to \overline{\mathbb Q}_\ell^*$ be the character 
such that $\chi=\chi_0\circ \mathrm N_{B/\mathbb F_q}$. Then 
we have $R\mathbf N_!\mathcal K_\chi=\mathcal K_{\chi_0}\otimes R\mathbf N_!\overline{\mathbb Q}_\ell.$ 
\end{enumerate}
\end{lemma}

\begin{proof}  (i) Let $q_1$ and $q_2$ be the projections  of $\mathrm{Res}(\mathbb G_{m, B})\times \mathbb G_{m, \mathbb F_q}$ 
to its factors. We have 
\begin{eqnarray*}
R(\mathbf Np_1i)_!(i^*p_1^*\mathcal K_\chi \otimes i^*f^*\mathcal L_\psi)
&\cong& R(\mathbf Nq_1)_!(q_1^*\mathcal K_\chi \otimes q_2^*(\mathcal L_\psi|_{\mathbb G_{m,\mathbb F_q}}))\\
&\cong& R\mathbf N_!(\mathcal K_\chi \otimes Rq_{1!}q_2^*(\mathcal L_\psi|_{\mathbb G_{m,\mathbb F_q}}))
\cong R\mathbf N_! \mathcal K_\chi[-1].
\end{eqnarray*}
Here we use the fact that $$Rq_{1!}q_2^*(\mathcal L_\psi|_{\mathbb G_{m,\mathbb F_q}})\cong \overline{\mathbb Q}_\ell[-1].$$
This follows from the proper base change theorem, the distinguished triangle 
$$R\Gamma_c(\mathbb G_{m, \overline{\mathbb F}_q}, \mathcal L_\psi)\to 
R\Gamma_c(\mathbb G_{a, \overline{\mathbb F}_q}, \mathcal L_\psi)\to \mathcal L_{\psi, \bar 0}\to.$$
and the fact (\cite[sommes trig., Th\'eor\`eme 2.7*]{SGA4.5})
$$R\Gamma_c(\mathbb G_{a, \overline{\mathbb F}_q}, \mathcal L_\psi)=0.$$

(ii) To prove $R\mathbf N_! \mathcal K_\chi=0$,  
we may make a base change to the $\overline{\mathbb F}_q$. Then for any 
$\mu\in \mathbb G_m(\overline{\mathbb F}_q)$, we may find some $\mu'\in \mathrm{Res}(\mathbb G_{m,B})(\overline{\mathbb F}_q)$ 
such that $\mathbf N(\mu')=\mu$. Denote the base change of $\mathrm{Res}(\mathbb G_{m, B})$ by 
$\mathrm{Res}(\mathbb G_{m, B})_{\overline{\mathbb F}_q}$. We have a commutative diagram 
$$\begin{tikzcd}
\mathrm{Res}(\mathbb G_{m, B})_{\overline{\mathbb F}_q}\arrow[r, "\mu'"]\arrow[d,"\mathbf N"]&
\mathrm{Res}(\mathbb G_{m, B})_{\overline{\mathbb F}_q}\arrow[d, "\mathbf N"]\\
\mathbb G_{m, \overline{\mathbb F}_q}\arrow[r, "\mu"]&
\mathbb G_{m, \overline{\mathbb F}_q},
\end{tikzcd}$$
where the horizontal arrows are the multiplications by $\mu'$ and $\mu$, respectively. It follows that 
\begin{eqnarray}\label{translation}
\mu^*R\mathbf N_! \mathcal K_\chi\cong R\mathbf N_!(\mu'^*\mathcal K_\chi).
\end{eqnarray}
Let $$m, \pi_1, \pi_2: \mathrm{Res}(\mathbb G_{m, B})\times \mathrm{Res}(\mathbb G_{m, B})\to \mathrm{Res}(\mathbb G_{m, B})$$
be the multiplication and the projections. By \cite[1.1.3.2]{L}, we have
$$m^*\mathcal K_\chi\cong \pi_1^*\mathcal K_\chi\otimes \pi_2^* \mathcal K_\chi.$$  
It follows that 
$$\mu'^* \mathcal K_\chi
\cong \mathcal K_{\chi,\bar \mu'} \otimes \mathcal K_\chi.$$
Combined with (\ref{translation}), we get
$$\mu^*R\mathbf N_! \mathcal K_\chi\cong \mathcal K_{\chi,\bar \mu'}\otimes R\mathbf N_!\mathcal K_\chi.$$
We thus have 
$$(R\mathbf N_! \mathcal K_\chi)_{\bar \mu} \cong \mathcal K_{\chi,\bar \mu'}\otimes (R\mathbf N_!\mathcal K_\chi)_{\bar 1}.$$
To prove $RN_!\mathcal K_\chi=0$, it suffices to prove  $(R\mathbf N_!\mathcal K_\chi)_{\bar 1}=0$. 
But 
$$(R\mathbf N_! \mathcal K_\chi)_{\bar 1}\cong R\Gamma_c\Big(\mathrm{ker}(\mathbf N)\otimes{\overline{\mathbb F}_q}, \mathcal K_\chi\Big).$$
Note that $\mathcal K_\chi|_{\mathrm{ker}(\mathbf N)}$ is the Lang sheaf on $\mathrm{ker}(\mathbf N)$ associated to the character
$$\chi: \{x\in B^*:\, N_{B/\mathbb F_q}(x)=1\}\to \overline{\mathbb Q}_\ell^*.$$
Our assertion follows from \cite[sommes trig., Th\'eor\`eme 2.7*]{SGA4.5}.

(iii) We have $\mathcal K_\chi=\mathbf N^*\mathcal K_{\chi_0}$ by \cite[sommes trig., (1.7.4)]{SGA4.5}. Our assertion follows from the projection 
formula. 
\end{proof}

\begin{lemma}\label{lemma3} Let $n+1=\mathrm{dim}_{\mathbb F_q}B$. Suppose $(p, n+1)=1$. There exists a mixed lisse
$\overline{\mathbb Q}_\ell$-sheaf $\mathcal G$ on $\mathbb G_{m,\mathbb F_q}$ of rank $2n+2$ with weights $\leq n+2$
such that 
$$R (\mathbf Np_1j)_!(j^*p_1^*\mathcal K_\chi \otimes j^*f^*\mathcal L_\psi)=\mathcal G[-(n+2)].$$ 
\end{lemma} 

\begin{proof} 
The vanishing, being lisse, the rank, and the weights of a sheaf are not affected by a finite base change. 
Choose a sufficiently large finite extension $\mathbb F_{q^d}$ of $\mathbb F_q$ so that 
$$B\otimes_{\mathbb F_q}\mathbb F_{q^d}\cong \mathbb F_{q^d}^{n+1}. $$
Let $\chi_i: \mathbb F_{q^d}^*\to\overline{\mathbb Q}_\ell^*$ be multiplicative characters so that 
the character 
$$(\mathbb F_{q^d}^*)^{n+1}\stackrel\cong \to (B\otimes_{\mathbb F_q}\mathbb F_{q^d})^*
\xrightarrow{\mathrm N_{B\otimes_{\mathbb F_q}\mathbb F_{q^d}/B}} B^* \stackrel\chi \to \overline{\mathbb Q}_\ell^*$$
is of the form 
$$(x_1, \ldots, x_{n+1})\mapsto \chi_1(x_1)\cdots\chi_{n+1}(x_{n+1}).$$
Denote by $\mathcal K_{\chi_i}$ the Kummer sheaf (Lang sheaf) on $\mathbb G_{m, \mathbb F_{q^d}}$ associated to the 
character $\chi_i$. 
By \ref{free}, the base changes of $\mathbf{N}$ and $\mathbf{Tr}$ to $\mathbb F_{q^d}$ are
\begin{eqnarray*}
Tr: \mathbb G^{n+1}_{m,\mathbb F_{q^d}}\to  \mathbb G_{a,\mathbb F_{q^d}},&\quad& (x_1, \ldots, x_{n+1})\to x_1+\cdots +x_{n+1},\\
N: \mathbb G^{n+1}_{m,\mathbb F_{q^d}}\to  \mathbb G_{m,\mathbb F_{q^d}},&\quad& (x_1, \ldots, x_{n+1})\to x_1\cdots x_{n+1}.
\end{eqnarray*}
Let $\pi_i: \mathbb G^{n+3}_{m,\mathbb F_{q^d}}\to \mathbb G_{m,\mathbb F_{q^d}}$ $(i=1, \ldots, n+3)$
be the projections to its factors, let $\mathrm{pr}:  
\mathbb G^{n+3}_{m,\mathbb F_{q^d}}\to \mathbb G^{n+1}_{m,\mathbb F_{q^d}}$ be the projection to the first $n+1$ factors, 
and let $f'$ be the morphism
$$f': \mathbb G^{n+3}_{m,\mathbb F_{q^d}}\to \mathbb G_{a,\mathbb F_{q^d}},\quad 
(x_1, \ldots, x_{n+1}, y, z)\mapsto y+z-yz(x_1+\cdots+x_{n+1}).$$
We have 
\begin{eqnarray}\label{a}
&&\Big(R(\mathbf Np_1j)_!(j^*p_1^*\mathcal K_\chi \otimes j^*f^*\mathcal L_\psi)\Big)\Big|_{\mathbb G_{m,  \mathbb F_{q^d}}}\\
&&\cong R (N\mathrm{pr})_! \Big(\pi_1 ^*\mathcal K_{\chi_1}\otimes \cdots\otimes \pi_{n+1}^*\mathcal K_{\chi_{n+1}}
\otimes f'^* \mathcal L_{\psi\circ\mathrm{Tr}_{\mathbb F_{q^d}/\mathbb F_q}}\Big).\nonumber
\end{eqnarray}
Let $\sigma$ be the automorphism 
$$\sigma: \mathbb G^{n+3}_{m,\mathbb F_{q^d}}\to  \mathbb G^{n+3}_{m,\mathbb F_{q^d}}, \quad 
(x_1, \ldots, x_{n+1}, y, z)\mapsto (x_1, \ldots, x_n, y, z,x_1\cdots x_{n+1})$$ and let $\tilde f$ be the morphism 
$$\tilde f: \mathbb G^{n+3}_{m,\mathbb F_{q^d}}\to \mathbb G_{a,\mathbb F_{q^d}},\quad 
(x_1, \ldots, x_n, y, z, w)\mapsto y+z-yz\Big(x_1+\cdots+x_n+ \frac{w}{x_1\cdots x_n}\Big).$$
We have 
\begin{eqnarray*}
\pi_i\sigma^{-1}=\pi_i\;(i=1, \ldots, n), \quad f' \sigma^{-1}=\tilde f,\quad N\mathrm{pr}=\pi_{n+3}\sigma.
\end{eqnarray*}
Moreover, we have
$$\pi_{n+1}\sigma^{-1}(x_1, \ldots, x_n, y, z, w)=\frac{w}{x_1\cdots x_n}.$$ 
It follows that 
$$(\pi_{n+1}\sigma^{-1})^*\mathcal K_{\chi_{n+1}}=\pi_{n+3}^*\mathcal K_{\chi_{n+1}}\otimes 
\pi_1^*\mathcal K_{\chi^{-1}_{n+1}}\otimes\cdots\otimes
\pi_n^*\mathcal K_{\chi^{-1}_{n+1}}.$$
So we have 
\begin{eqnarray}\label{b}
&&R (N\mathrm{pr})_! \Big(\pi_1 ^*\mathcal K_{\chi_1}\otimes \cdots\otimes \pi_{n+1}^*\mathcal K_{\chi_{n+1}}
\otimes f'^* \mathcal L_{\psi\circ\mathrm{Tr}_{\mathbb F_{q^d}/\mathbb F_q}}\Big)\\
&\cong& R (\pi_{n+3}\sigma)_! \Big(\pi_1 ^*\mathcal K_{\chi_1}\otimes \cdots\otimes \pi_{n+1}^*\mathcal K_{\chi_{n+1}}
\otimes f'^* \mathcal L_{\psi\circ\mathrm{Tr}_{\mathbb F_{q^d}/\mathbb F_q}}\Big)\nonumber\\
&\cong& R \pi_{n+3,!}\sigma^{-1, *} \Big(\pi_1 ^*\mathcal K_{\chi_1}\otimes \cdots\otimes \pi_{n+1}^*\mathcal K_{\chi_{n+1}}
\otimes f'^* \mathcal L_{\psi\circ\mathrm{Tr}_{\mathbb F_{q^d}/\mathbb F_q}}\Big)\nonumber\\
&\cong& R \pi_{n+3,!}\Big(
\pi_1 ^*\mathcal K_{\chi_1\chi_{n+1}^{-1}}\otimes \cdots\otimes \pi_{n}^*\mathcal K_{\chi_n\chi_{n+1}^{-1}}\otimes 
\pi_{n+3}^*\mathcal K_{\chi_{n+1}}
\otimes \tilde f^* \mathcal L_{\psi\circ\mathrm{Tr}_{\mathbb F_{q^d}/\mathbb F_q}}\Big)\nonumber\\
&\cong& \mathcal K_{\chi_{n+1}}\otimes  R \pi_{n+3,!}
\Big(\pi_1 ^*\mathcal K_{\chi_1\chi_{n+1}^{-1}}\otimes \cdots\otimes \pi_{n}^*\mathcal K_{\chi_n\chi_{n+1}^{-1}}
\otimes \tilde f^* \mathcal L_{\psi\circ\mathrm{Tr}_{\mathbb F_{q^d}/\mathbb F_q}}\Big).\nonumber
\end{eqnarray}
Consider the family of Laurent polynomials 
$$F=a_1y+a_2z+a_3 x_1yz+ \cdots+ a_{n+2} x_nyz +a_{n+3} \frac{yz}{x_1\cdots x_n}$$
in the variables $x_1, \ldots, x_n, y, z$. It defines a morphism 
\begin{eqnarray*}
F: \mathbb A^{n+3}_{\mathbb F_q} \times \mathbb G_{m,\mathbb F_q}^{n+2}&\to& \mathbb A^1, \\
(a_1, \ldots, a_{n+3}, x_1, \ldots, x_n, y, z)&\mapsto& 
a_1y+a_2z+a_3 x_1yz+ \cdots+ a_{n+2} x_nyz +a_{n+3} \frac{yz}{x_1\cdots x_n}.
\end{eqnarray*}
Let $\mathrm{pr}_1:\mathbb A^{n+3}_{\mathbb F_q} \times \mathbb G_{m,\mathbb F_q}^{n+2}\to \mathbb A^{n+3}_{\mathbb F_q}$ 
be the projection to the 
first factor, and for every $i\in\{1, \ldots, n\}$, let 
$$q_i:  \mathbb A^{n+3}_{\mathbb F_q} \times \mathbb G_{m,\mathbb F_q}^{n+2}\to   \mathbb G_{m,\mathbb F_q},\quad
(a_1, \ldots, a_{n+3}, x_1, \ldots, x_n, y, z)\mapsto x_i$$ be the projection. 
According to \cite{F}, we have the GKZ hypergeometric complex 
$$\mathrm{Hyp}=R\mathrm{pr}_{1,!} \Big(q_1 ^*\mathcal K_{\chi_1\chi_{n+1}^{-1}}\otimes \cdots\otimes q_{n}^*\mathcal K_{\chi_n\chi_{n+1}^{-1}}
\otimes F^* \mathcal L_\psi\Big)[2n+5].$$
Let $h$ be the morphism
$$h: \mathbb G_{m,\mathbb F_q}\to \mathbb A^{n+3}_{\mathbb F_q},\quad w\mapsto 
(1,1,-1, \ldots, -1, -w).$$ 
Then by the proper base change theorem, we have 
\begin{eqnarray}\label{c}
&&R \pi_{n+3,!}
\Big(\pi_1 ^*\mathcal K_{\chi_1\chi_{n+1}^{-1}}\otimes \cdots\otimes \pi_{n}^*\mathcal K_{\chi_n\chi_{n+1}^{-1}}
\otimes \tilde f^* \mathcal L_{\psi\circ\mathrm{Tr}_{\mathbb F_{q^d}/\mathbb F_q}}\Big)
\cong h^*\mathrm{Hyp}[-(2n+5)]|_{\mathbb G_{m,\mathbb F_q^d}}.
\end{eqnarray}
Let $\Delta_\infty(F)$ be the Newton polytope of the Laurent polynomial $F$. In \cite[Proposition 3.1]{LW}, 
Lin and Wan describe $\Delta_\infty(F)$ in detail. They show that if $(p,n+1)=1$, then the image of $h$ is contained in the 
non-degenerate locus $V\subset  \mathbb A^{n+3}_{\mathbb F_q}$  
of $F$, and $(n+2)!\mathrm{vol}(\Delta_\infty(F))=2n+2$. Using the explicit description of the codimension 1 faces of 
$\Delta_\infty(F)$ in \cite[3.2]{LW}, one verifies the condition of \cite[Theorem 0.4 (iii)]{F} holds. 
So by \cite[Theorems 0.3 (i), 0.4 (i), (iii)]{F}, on the non-degenerate locus $V$ 
of $F$, 
there exists a mixed lisse $\overline{\mathbb Q}_\ell$-sheaf $\mathcal H$ of rank $2n+2$ 
with weights $\leq n+2$ such that 
\begin{eqnarray}\label{d}
\mathrm{Hyp}|_V \cong \mathcal H[n+3].
\end{eqnarray}  
Combining (\ref{a})-(\ref{d}) together, we get that 
\begin{eqnarray*}
\Big(R(\mathbf Np_1j)_!(j^*p_1^*\mathcal K_\chi \otimes j^*f^*\mathcal L_\psi)\Big)\Big|_{\mathbb G_{m,  \mathbb F_{q^d}}}
\cong\mathcal K_{\chi_{n+1}}\otimes h^*\mathcal H[-(n+2)]|_{\mathbb G_{m,\mathbb F_q^d}}.
\end{eqnarray*}
The lemma then follows.
\end{proof}

\begin{lemma}\label{lemma4} Suppose $p|(n+1)$ and $\psi=\psi_0\circ\mathrm{Tr}_{\mathbb F_q/\mathbb F_p}$ for 
some non-trivial character $\psi_0$ of $\mathbb F_p$. Write $n+1=p^km\geq 2m$ such that $(p, m)=1$. Let 
$$\mathcal G=R (\mathbf Np_1j)_!(j^*p_1^*\mathcal K_\chi \otimes j^*f^*\mathcal L_\psi)[n+2].$$
\begin{enumerate}[(i)]
\item If $n+1>2m$, then $\mathcal G$ is 
a lisse $\overline{\mathbb Q}_\ell$-sheaf on $\mathbb G_{m,\mathbb F_q}$ of rank $n+1$ 
with weights $\leq n+2$.
\item If $n+1=2m$, then 
$\mathcal G|_{\mathbb G_{m, \mathbb F_q}-\{m^{-(n+1)}\}}$
is a mixed lisse
$\overline{\mathbb Q}_\ell$-sheaf $\mathcal G$ on $\mathbb G_{m, \mathbb F_q}-\{m^{-(n+1)}\}$ 
of rank $n+1$ with weights $\leq n+2$.
\end{enumerate}
\end{lemma} 

\begin{proof} The Frobenius 
correspondence 
$$\mathrm{Fr}: \mathbb G_{m,\mathbb F_{q^d}} \to  \mathbb G_{m,\mathbb F_{q^d}}, \quad x\mapsto x^{p^k}$$
is a finite radiciel morphism and has no effect on \'etale topology. Combined with (\ref{b}) it suffices to show the sheaves 
\begin{eqnarray}\label{o}
&&\mathrm{Fr}^* 
\Big(R (\mathbf Np_1j)_!(j^*p_1^*\mathcal K_\chi \otimes j^*f^*\mathcal L_\psi)\Big)\Big|_{\mathbb G_{m,\mathbb F_{q^d}}}\\
&\cong& \mathrm{Fr}^* \mathcal K_{n+1}\otimes \mathrm{Fr}^* R \pi_{n+3,!}
\Big(\pi_1 ^*\mathcal K_{\chi_1\chi_{n+1}^{-1}}\otimes \cdots\otimes \pi_{n}^*\mathcal K_{\chi_n\chi_{n+1}^{-1}}
\otimes \tilde f^* \mathcal L_{\psi\circ\mathrm{Tr}_{\mathbb F_{q^d}/\mathbb F_q}}\Big)\nonumber
\end{eqnarray} 
have the same property as stated in the lemma. (Note that since $m\in \mathbb Z$, the point $m^{-(n+1)}$ is fixed under $\mathrm{Fr}$.) 
By the proper base change theorem, we have an isomorphism 
\begin{eqnarray}\label{i}
&&\mathrm{Fr}^* R \pi_{n+3,!}
\Big(\pi_1 ^*\mathcal K_{\chi_1\chi_{n+1}^{-1}}\otimes \cdots\otimes \pi_{n}^*\mathcal K_{\chi_n\chi_{n+1}^{-1}}
\otimes \tilde f^* \mathcal L_{\psi\circ\mathrm{Tr}_{\mathbb F_{q^d}/\mathbb F_q}}\Big)\\
&\cong& R \pi_{n+3,!} \Big(\pi_1 ^*\mathcal K_{\chi_1\chi_{n+1}^{-1}}\otimes \cdots\otimes \pi_{n}^*\mathcal K_{\chi_n\chi_{n+1}^{-1}}
\otimes \tilde f'^* \mathcal L_{\psi\circ\mathrm{Tr}_{\mathbb F_{q^d}/\mathbb F_q}}\Big),\nonumber
\end{eqnarray}
where $\tilde f'$ is the morphism 
$$\tilde f': \mathbb G_{m,\mathbb F_{q^d}}^{n+3} \to \mathbb G_{a,\mathbb F_{q^d}},
\quad (x_1,\ldots, x_n, y, z, w)\mapsto y+z-yz\Big(x_1+\cdots+x_n+\frac{w^{p^k}}{x_1\cdots x_n}\Big).$$
Let $\tau$ be the automorphism 
$$\mathbb G_{m,\mathbb F_{q^d}}^{n+3} \to \mathbb G_{m,\mathbb F_{q^d}}^{n+3},
\quad (x_1,\ldots, x_n, y, z, w)\mapsto (x_1 z^{-1}, \ldots, x_nz^{-1}, y, y^{-1}z, w).$$ 
We have 
\begin{eqnarray}\label{ii}
&&R \pi_{n+3,!} \Big(\pi_1 ^*\mathcal K_{\chi_1\chi_{n+1}^{-1}}\otimes \cdots\otimes \pi_{n}^*\mathcal K_{\chi_n\chi_{n+1}^{-1}}
\otimes\tilde  f'^* \mathcal L_{\psi\circ\mathrm{Tr}_{\mathbb F_{q^d}/\mathbb F_q}}\Big)\\
&\cong& R( \pi_{n+3}\tau)_! \tau^*\Big(\pi_1 ^*\mathcal K_{\chi_1\chi_{n+1}^{-1}}\otimes \cdots\otimes \pi_{n}^*\mathcal K_{\chi_n\chi_{n+1}^{-1}}
\otimes\tilde  f'^* \mathcal L_{\psi\circ\mathrm{Tr}_{\mathbb F_{q^d}/\mathbb F_q}}\Big)\nonumber\\
&\cong& R\pi_{n+3,!} \Big(\pi_1 ^*\mathcal K_{\chi_1\chi_{n+1}^{-1}}\otimes \cdots\otimes \pi_{n}^*\mathcal K_{\chi_n\chi_{n+1}^{-1}}
\otimes \pi_{n+2}^* \mathcal K_{\chi_1^{-1}\cdots\chi_n^{-1}\chi^n_{n+1}}
\otimes \tilde f''^* \mathcal L_{\psi\circ\mathrm{Tr}_{\mathbb F_{q^d}/\mathbb F_q}}\Big),\nonumber
\end{eqnarray}
where $\tilde f''=\tilde f'\tau$ is the morphism 
$$\tilde f'': \mathbb G_{m,\mathbb F_{q^d}}^{n+3} \to \mathbb G_{a,\mathbb F_{q^d}},
\quad (x_1,\ldots, x_n, y, z, w)\mapsto y+y^{-1}z- \Big(x_1+\cdots+x_n+\frac{w^{p^k}z^{n+1}}{x_1\cdots x_n}\Big).$$
Let $\mathrm{Fr}'$ and $\tilde f'''$ be the morphism
\begin{eqnarray*}
\mathrm{Fr}':  \mathbb G_{m,\mathbb F_{q^d}}^{n+3}\to  \mathbb G_{m,\mathbb F_{q^d}}^{n+3},&&
(x_1,\ldots, x_n, y, z, w)\mapsto (x_1^{p^k},\ldots, x_n^{p^k},y, z, w),\\
\tilde f''': \mathbb G_{m,\mathbb F_{q^d}}^{n+3} \to \mathbb G_{a,\mathbb F_{q^d}},
&& (x_1,\ldots, x_n, y, z, w)\mapsto y+y^{-1}z- \Big(x_1^{p^k}+\cdots+x_n^{p^k}+\Big(\frac{w z^m}{x_1\cdots x_n}\Big)^{p^k}\Big).
\end{eqnarray*} 
Then $\mathrm{Fr}'$ is a finite raciel homomorphism of group schemes and $\tilde f''\mathrm{Fr}'=\tilde f'''$. So we have
\begin{eqnarray}\label{iii}
&& R\pi_{n+3,!} \Big(\pi_1 ^*\mathcal K_{\chi_1\chi_{n+1}^{-1}}\otimes \cdots\otimes \pi_{n}^*\mathcal K_{\chi_n\chi_{n+1}^{-1}}
\otimes \pi_{n+2}^* \mathcal K_{\chi_1^{-1}\cdots\chi_n^{-1}\chi^n_{n+1}}
\otimes \tilde f''^* \mathcal L_{\psi\circ\mathrm{Tr}_{\mathbb F_{q^d}/\mathbb F_q}}\Big)\\
&\cong& R(\pi_{n+3}\mathrm{Fr}')_!\mathrm{Fr}'^* \Big(\pi_1 ^*\mathcal K_{\chi_1\chi_{n+1}^{-1}}\otimes \cdots\otimes \pi_{n}^*\mathcal K_{\chi_n\chi_{n+1}^{-1}}
\otimes \pi_{n+2}^* \mathcal K_{\chi_1^{-1}\cdots\chi_n^{-1}\chi^n_{n+1}}
\otimes \tilde f''^* \mathcal L_{\psi\circ\mathrm{Tr}_{\mathbb F_{q^d}/\mathbb F_q}}\Big)\nonumber\\
&\cong& R\pi_{n+3,!}\Big(\pi_1 ^*\mathcal K_{\chi^{p^k}_1\chi_{n+1}^{-p^k}}\otimes \cdots\otimes 
\pi_{n}^*\mathcal K_{\chi^{p^k}_n\chi_{n+1}^{-p^k}}
\otimes \pi_{n+2}^* \mathcal K_{\chi_1^{-1}\cdots\chi_n^{-1}\chi^n_{n+1}}
\otimes \tilde f'''^* \mathcal L_{\psi\circ\mathrm{Tr}_{\mathbb F_{q^d}/\mathbb F_q}}\Big).\nonumber
\end{eqnarray}
Finally, let $\hat f$ be the morphism 
$$\hat f: \mathbb G_{m,\mathbb F_{q^d}}^{n+3} \to \mathbb G_{a,\mathbb F_{q^d}},
\quad (x_1,\ldots, x_n, y, z, w)\mapsto y+y^{-1}z- \Big(x_1+\cdots+x_n+\frac{w z^m}{x_1\cdots x_n}\Big).$$
We claim that
\begin{eqnarray}\label{iv}
\tilde f'''^* \mathcal L_{\psi\circ\mathrm{Tr}_{\mathbb F_{q^d}/\mathbb F_q}}
\cong  \hat f^* \mathcal L_{\psi\circ\mathrm{Tr}_{\mathbb F_{q^d}/\mathbb F_q}}
\end{eqnarray}
We have $\mathcal L_{\psi \circ\mathrm{Tr}_{\mathbb F_{q^d}/\mathbb F_q}}
\cong \mathcal L_{\psi_0\circ\mathrm{Tr}_{\mathbb F_{q^d}/\mathbb F_p}}$. The sheaf
$\mathcal L_{\psi_0\circ\mathrm{Tr}_{\mathbb F_{q^d}/\mathbb F_p}}$ is obtained by pushing forward the Artin-Schreier torsor
$$\mathbb G_{a, \mathbb F_{q^d}}\to \mathbb G_{a, \mathbb F_{q^d}}, \quad x\mapsto x^p-x$$
via the character $\psi_0$. To prove the claim, it suffices to show the inverse images of the above Artin-Schreier torsor by $\tilde f'''$ and 
$\hat f$ are isomorphic. This follows from the fact that we have $\tilde f'''-\hat f= h^p-h$ 
for some Laurent polynomial $h$ and we have an exact sequence 
$$\Gamma(X, \mathcal O_X)\stackrel{h\mapsto h^p-h}\to \Gamma(X, \mathcal O_X)\to  H^1(X, \mathbb F_p)$$ in the 
Artin-Schreier theory. 
Combining (\ref{i})-(\ref{iv}) together, we get 
\begin{eqnarray}\label{v}
&&\mathrm{Fr}^* R \pi_{n+3,!}
\Big(\pi_1 ^*\mathcal K_{\chi_1\chi_{n+1}^{-1}}\otimes \cdots\otimes \pi_{n}^*\mathcal K_{\chi_n\chi_{n+1}^{-1}}
\otimes \tilde f^* \mathcal L_{\psi\circ\mathrm{Tr}_{\mathbb F_{q^d}/\mathbb F_q}}\Big)\\
&\cong&R\pi_{n+3,!}\Big(\pi_1 ^*\mathcal K_{\chi^{p^k}_1\chi_{n+1}^{-p^k}}\otimes \cdots\otimes 
\pi_{n}^*\mathcal K_{\chi^{p^k}_n\chi_{n+1}^{-p^k}}
\otimes \pi_{n+2}^* \mathcal K_{\chi_1^{-1}\cdots\chi_n^{-1}\chi^n_{n+1}}
\otimes \hat f^* \mathcal L_{\psi\circ\mathrm{Tr}_{\mathbb F_{q^d}/\mathbb F_q}}\Big).\nonumber
\end{eqnarray}
Consider the family of Laurent polynomials 
$$\hat F=a_1y+a_2y^{-1}z+a_3 x_1+ \cdots+ a_{n+2} x_n +a_{n+3} \frac{z^m}{x_1\cdots x_n}$$
in the variables $x_1, \ldots, x_n, y, z$. It defines a morphism 
\begin{eqnarray*}
\hat F: \mathbb A^{n+3}_{\mathbb F_q} \times \mathbb G_{m,\mathbb F_q}^{n+2}&\to& \mathbb A^1, \\
(a_1, \ldots, a_{n+3}, x_1, \ldots, x_n, y, z)&\mapsto& 
a_1y+a_2y^{-1}z+a_3 x_1+ \cdots+ a_{n+2} x_n +a_{n+3} \frac{z^m}{x_1\cdots x_n}.
\end{eqnarray*}
Let $\mathrm{pr}_1:\mathbb A^{n+3}_{\mathbb F_q} \times \mathbb G_{m,\mathbb F_q}^{n+2}\to \mathbb A^{n+3}_{\mathbb F_q}$ 
be the projection to the 
first factor, and for $i\in\{1, \ldots, n+2\}$, let 
$q_i:  \mathbb A^{n+3}_{\mathbb F_q} \times \mathbb G_{m,\mathbb F_q}^{n+2}\to   \mathbb G_{m,\mathbb F_q}$ be the projections
to the $n+2$ factors $\mathbb G_{m,\mathbb F_q}$. 
According to \cite{F}, we have the GKZ hypergeometric complex 
$$\mathrm{Hyp}=R\mathrm{pr}_{1,!} \Big(q_1 ^*\mathcal K_{\chi_1^{p^k}\chi_{n+1}^{-p^k}}\otimes \cdots
\otimes q_{n}^*\mathcal K_{\chi_n^{p^k}\chi_{n+1}^{-p^k}}\otimes q_{n+2}^*\mathcal K_{\chi_1^{-1}\cdots\chi_n^{-1}\chi_{n+1}^n}
\otimes \hat F^* \mathcal L_\psi\Big)[2n+5].$$
Let $h$ be the morphism
$$h: \mathbb G_{m,\mathbb F_q}\to \mathbb A^{n+3}_{\mathbb F_q},\quad w\mapsto 
(1,1,-1, \ldots, -1, -w).$$ 
Then by the proper base change theorem, we have 
\begin{eqnarray}\label{last}
&&R\pi_{n+3,!}\Big(\pi_1 ^*\mathcal K_{\chi^{p^k}_1\chi_{n+1}^{-p^k}}\otimes \cdots\otimes 
\pi_{n}^*\mathcal K_{\chi^{p^k}_n\chi_{n+1}^{-p^k}}
\otimes \pi_{n+2}^* \mathcal K_{\chi_1^{-1}\cdots\chi_n^{-1}\chi^n_{n+1}}
\otimes \hat f^* \mathcal L_{\psi\circ\mathrm{Tr}_{\mathbb F_{q^d}/\mathbb F_q}}\Big)\\
&\cong& h^*\mathrm{Hyp}[-(2n+5)]|_{\mathbb G_{m,\mathbb F_q^d}}.\nonumber
\end{eqnarray}
Let $\Delta_\infty(\hat F)$ be the Newton polytope of the Laurent polynomial $\hat F$. By Lemma \ref{p|n+1} below, since 
$n+1=p^km \geq 2m$, we have
$$(n+2)!\mathrm{vol}(\Delta_\infty(\hat F))=n+1.$$ By the explicit description of the codimension 1 faces of 
$\Delta_\infty(\hat F)$ in the proof of Lemma \ref{p|n+1} (ii)-(iii), one verifies the condition of \cite[Theorem 0.4 (iii)]{F} holds. 
So by \cite[Theorems 0.3 (i), 0.4 (i), (iii)]{F}, on the non-degenerate locus $V\subset  \mathbb A^{n+3}_{\mathbb F_q}$ 
of $\hat F$, 
there exists a mixed lisse $\overline{\mathbb Q}_\ell$-sheaf $\mathcal H$ of rank $n+1$ 
with weights $\leq n+2$ such that 
\begin{eqnarray}\label{eqnhyp}
\mathrm{Hyp}|_V \cong \mathcal H[n+3].
\end{eqnarray}  
If $n+1>2m$, then by Lemma \ref{p|n+1} (iii), 
the image of $h$ is contained in the 
non-degenerate locus $V$. Combining (\ref{o}), (\ref{v})-(\ref{eqnhyp}) together, we get 
\begin{eqnarray*}
\mathrm{Fr}^*\Big(R(\mathbf Np_1j)_!(j^*p_1^*\mathcal K_\chi \otimes j^*f^*\mathcal L_\psi)\Big)\Big|_{\mathbb G_{m,  \mathbb F_{q^d}}}
\cong\mathrm{Fr}^* \mathcal K_{\chi_{n+1}}\otimes h^*\mathcal H[-(n+2)]|_{\mathbb G_{m,\mathbb F_q^d}}.
\end{eqnarray*}
This proves the lemma in the case $n+1>2m$. If $n+1=2m$, then by Lemma \ref{p|n+1} (ii), the image of $\mathbb G_{m, \mathbb F_{q^d}}
-\{m^{-(n+1)}\}$ under $h$ is contained in the non-degenerate locus $V$. We prove 
the lemma in this case exactly as above.  
\end{proof}

\begin{remark} The condition (iii) in \cite[Theorem 0.4]{F} actually follows from the condition that the Laurent polynomial 
$f=\sum_{j=1}^N a_j t_1^{w_{1j}}\cdots t_n^{w_{nj}}$ is nondegenerate with respect to the Newton polytope $\Delta_\infty$
at $\infty$, and hence is redundant, where 
$\Delta_\infty$ is the convex hull of $0$ and $\mathbf w_j=(w_{1j}, \ldots, w_{nj})$ $(j=1, \ldots, N)$ in $\mathbb R^n$. This can be seen as follows. 
Suppose $f$ is nondegenerate but the condition (iii) in \cite[Theorem 0.4]{F} does not hold. Let $\Gamma$ be a codimension 
1 face of $\Delta_\infty$ not containing the origin, and let 
$$\phi(v_1, \ldots, v_n)=d_1v_1+\cdots +d_n v_n$$
be a linear form such that $d_1, \ldots, d_n$ are relatively prime integers, that $\phi|_{\Delta_\infty}$ takes its minimum $d_\Gamma$ exactly on
$\Gamma$, but $d_\Gamma=pd$ for some integer $d$. One can check $\mathbb Z^n/\mathbb Z(d_1, \ldots, d_n)$ 
is torsion free and hence free. So $\mathbb Z(d_1, \ldots, d_n)$ is a direct factor of $\mathbb Z^n$. Changing the basis of $\mathbb Z^n$, 
we may assume $$(d_1, \ldots, d_n)=(1, 0,\ldots, 0).$$ Let $f_\Gamma=\sum_{\mathbf w_j\in \Gamma}a_j t_1^{w_{1j}}\cdots t_n^{w_{nj}}.$
Then $f_\Gamma$ is of the form
$$f_\Gamma=t_1^{pd} g(t_2, \ldots, t_n).$$ 
The assumption that $f$ is nondegenerate implies that 
$$g_1:=t_1^{pd} g(t_2, \ldots, t_n),\quad g_2:=t_1^{d} g(t_2, \ldots, t_n)$$ are both nondegenerate. 
Let $\Delta_\infty(g_1)$ and $\Delta_\infty(g_2)$ be the Newton polytopes at $\infty$ of $g_1$ and $g_2$, respectively. 
Then by \cite[Theorem 1.3 (b)]{DL}, we have 
\begin{eqnarray}\label{dimcount}
\mathrm{dim}\, H^n(\mathbb G^n_{m, \bar{\mathbb F}_q} , g_1^* \mathcal L_\psi)&=&n! \mathrm{vol}(\Delta_\infty(g_1)),\nonumber\\
\mathrm{dim}\, H^n(\mathbb G^n_{m, \bar{\mathbb F}_q} , g_2^* \mathcal L_\psi)&=&n! \mathrm{vol}(\Delta_\infty(g_2)),\nonumber\\
\mathrm{vol}(\Delta_\infty(g_1))&=&p\,\mathrm{vol}(\Delta_\infty(g_2)). 
\end{eqnarray}
Let $G$ be the finite radiciel morphism
$$G: \mathbb G_m^n\to \mathbb G_m^n, \quad (t_1,t_2, \ldots, t_n)\mapsto (t_1^p, t_2, \ldots, t_n).$$
We have $g_2\circ G=g_1$ and hence
\begin{eqnarray*}
 H^n(\mathbb G^n_{m, \bar{\mathbb F}_q} , g_1^* \mathcal L_\psi)\cong
 H^n(\mathbb G^n_{m, \bar{\mathbb F}_q} , G^*g_2^* \mathcal L_\psi)
 \cong   H^n(\mathbb G^n_{m, \bar{\mathbb F}_q} , g_2^* \mathcal L_\psi).
\end{eqnarray*}
This contradicts with (\ref{dimcount}). 
Alternatively, one may note that the exponential sums associated to $g_1$ and $g_2$ are equal, and hence they have 
the same $L$-function. If $g_1$ and $g_2$ are nondegenerate, the degrees of their $L$-functions are 
$(-1)^{n+1}n! \mathrm{vol}(\Delta_\infty(g_1))$ and $(-1)^{n+1}n!\mathrm{vol}(\Delta_\infty(g_2))$, respectively. Again this leads 
to contradiction. It would be interesting to give a proof without using the $\ell$-adic cohomology or $L$-functions. 
\end{remark}

\subsection{Proof of Theorem \ref{thm}} 
By Lemmas \ref{lemma1}-\ref{lemma4}, the distinguished triangle (\ref{triangle}) 
can be written as 
$$\mathcal G[-(n+2)]\to \mathrm{EIK}(-1)[-2]\to R\mathbf N_!\mathcal K_\chi[-1]\to,$$
We thus have a distinguished triangle
$$\mathcal G(1)[-n]\to \mathrm{EIK}\to R\mathbf N_!\mathcal K_\chi(1)[1]\to.$$
Theorem \ref{thm} (i) is proved in Lemma \ref{lemma2}. Theorem \ref{thm} (ii)-(iii) about the properties of $\mathcal H$ 
follows from the properties of $\mathcal G$ in Lemmas \ref{lemma3}-\ref{lemma4}.

\begin{lemma}\label{p|n+1} Let $m$ be a positive integer such that $(p, m)=1$. Consider the one parameter family of  Laurent polynomials  
$$\hat f=x_{n+1}+ \frac{x_{n+2}}{x_{n+1}}- \Big(x_1+\cdots+x_n+\frac{w x_{n+2}^m}{x_1\cdots x_n}\Big)$$
in the $(n+2)$ variables $x_1, \ldots, x_{n+2}$ with parameter $w$. 
Let $\Delta_\infty(\hat f)$ be its Newton polytope at $\infty$. 

(i) If $n+1<2m$, the Laurent polynomial $\hat f$ is non-degenerate for every $w\not=0$, and we have 
$$(n+2)!\mathrm{vol}(\Delta_\infty(\hat f))=2m.$$

(ii) If $n+1=2m$, the Laurent polynomial $\hat f$ is non-degenerate for every $w\not\in \{0, m^{-(n+1)}\}$, and we have 
$$(n+2)!\mathrm{vol}(\Delta_\infty(\hat f))=2m = n+1.$$

(iii) If $n+1>2m$,  the Laurent polynomial $\hat f$ is non-degenerate for every $w\not=0$, and we have 
$$(n+2)!\mathrm{vol}(\Delta_\infty(\hat f))=n+1.$$ 
\end{lemma}

\begin{proof}  Recall that $\Delta:=\Delta_\infty(\hat f)$ is the convex hull in $\mathbb R^{n+2}$ of $(0,\ldots, 0)$ together with the
exponents $\{ V_1, \cdots, V_{n+1}, V_{n+2}, V_{n+3}\}$ of monomials appeared in $\hat f$, where 
$$V_1=(1, 0, \cdots, 0), \cdots, V_{n+1}=(0, \cdots, 0, 1, 0), V_{n+2}=(0, \cdots, 0, -1, 1), 
V_{n+3}=(-1, \cdots, -1, 0, m).$$  
To find the non-degenerate locus of $\hat f$, the first step is to find all the codimension $1$ faces $\delta$ of $\Delta$ not containing the origin. The answer is a little subtle, and depends on the relative size between $n+1$ 
and $2m$. 

(i) Assume $n+1<2m$. Let $\delta$ be a codimension $1$ face of $\Delta$ not containing the origin in $\mathbb R^{n+2}$. Then
$\delta$ must contain at least $n+2$ points out of the $n+3$ points  $V_1, \cdots, V_{n+3}$. We claim that $\delta$ must 
contain $\{ V_1, \cdots, V_n, V_{n+3}\}$. 
If $\delta$ does not contain $V_{n+3}$, then $\delta$ must contain $\{V_1, \cdots, V_{n+2}\}$ 
and thus  the normalized equation of $\delta$ is
$$x_1 +\cdots + x_{n+1} + 2x_{n+2}=1.$$
Substituting the coordinates of $V_{n+3}$ into this equation, one checks that 
$$(-1)+\cdots +(-1)+ 0 + 2m = 2m - n > 1.$$
This means that $V_{n+3}$ and the origin are on different sides of $\delta$, 
contradicting to our assumption that $\delta$ is a codimension $1$ face of $\Delta$ not containing the origin.  
Similarly, if $\delta$ does not contain $V_1$, then $\delta$ must contain $\{V_2, \cdots, V_{n+3}\}$ 
and the normalized equation of $\delta$ is
$$(2m-n)x_1+x_2 +\cdots + x_{n+1} + 2x_{n+2}=1.$$
Substituting the coordinates of $V_{1}$ into this equation, one checks that 
$$(2m-n) + 0 +\cdots + 0+ 0 = 2m - n > 1.$$ Contradiction again.  By symmetry, we also have a contradiction if $\delta$ does not contain 
$V_i$ for some $1\leq i \leq n$. This proves the claim that $\delta$ must contain $\{ V_1, \cdots, V_n, V_{n+3}\}$. 
There are only two possibilities for $\delta$: either $\delta$ contains $V_{n+1}$ or $\delta$ contains $V_{n+2}$. 
If $\delta$ contains $\{V_1,  \cdots, V_n, V_{n+1}, V_{n+3}\}$, then the equation of $\delta$ is 
$$\delta_1: \ x_1+\cdots +x_{n+1} + \frac{n+1}{m} x_{n+2} = 1.$$
Substituting $V_{n+2}$ into this equation, one finds that 
$$0 + \cdots + 0 + (-1) + \frac{n+1}{m} = \frac{n+1 - m}{m} < 1$$
which proves that $\delta_1$ is indeed a codimension $1$ face of $\Delta$ not containing the origin. 
If $\delta$ contains $\{V_1,  \cdots, V_n, V_{n+2}, V_{n+3}\}$, then the equation of $\delta$ is 
$$\delta_2: \ x_1+\cdots +x_{n} + \frac{n+1-m}{m}x_{n+1} + \frac{n+1}{m} x_{n+2} = 1.$$
Substituting $V_{n+1}$ into this equation, one finds that 
$$0 + \cdots + 0  + \frac{n+1-m}{m}  + 0= \frac{n+1 - m}{m} < 1$$
which proves that $\delta_2$ is indeed a codimension $1$ face of $\Delta$ not containing the origin. 

Let $\hat f_{\delta}$ be the sum of those terms in $\hat f$ whose exponents lie in $\delta$. Then, we have 
$$\hat f_{\delta_1}=x_{n+1}- \Big(x_1+\cdots+x_n+\frac{w x_{n+2}^m}{x_1\cdots x_n}\Big),\quad 
\hat f_{\delta_2}=\frac{x_{n+2}}{x_{n+1}}- \Big(x_1+\cdots+x_n+\frac{w x_{n+2}^m}{x_1\cdots x_n}\Big).$$
Note that $\hat f_{\delta_i}$ $(i=1, 2)$
are diagonal in the sense that the number of non-constant monomials appeared in them coincide with the number of variables $n+2$. 
Calculating the determinants of the square matrices formed by the exponents appeared in $\hat f_{\delta_i}$, we get 
$$\mathrm{det}\left(\begin{array}{ccccc}
1&&&&\\
&\ddots&&&\\
&&1&&\\
&&&1&\\
-1&\ldots&-1& 0&m
\end{array}
\right)=m, \quad 
\mathrm{det}\left(\begin{array}{ccccc}
1&&&&\\
&\ddots&&&\\
&&1&&\\
&&&-1&1\\
-1&\ldots&-1& 0 &m
\end{array}
\right)=-m.$$
Since $m\not\equiv 0\mod p$, $\hat f_{\delta_i}$ are non-degenerate. It follows that $\hat f$ is non-degenerate.  
The cones with vertex $0$ over $\delta_1$ and $\delta_2$ give a triangulation of $\Delta$. We conclude that 
$$(n+2)!\mathrm{vol}(\Delta_\infty(\hat f))= m + m =2m.$$ 

(ii) Assume $n+1=2m$. In this case, the two equations for $\delta_1$ and $\delta_2$ in case (i) are 
the same. Thus, there is only one codimension $1$ face $\delta$ of $\Delta$ not containing the origin. 
The equation for $\delta$ is 
$$\delta: \ x_1+\cdots +x_{n+1} + 2x_{n+2} = 1.$$
This $\delta$ contains the whole $\{V_1, \cdots, V_{n+1}, V_{n+2}, V_{n+3}\}$ and $\hat f_{\delta}=\hat f$ 
is no longer diagonal. The same triangulation as in case (i) shows that 
$(n+2)!\mathrm{vol}(\Delta_\infty(\hat f))=2m$.  Suppose $\sigma$ is a face (of any dimension) of 
$\Delta_\infty(\hat f)$ not containing the origin. For non-degeneracy, 
we need to check that the system 
$$\frac{\partial \hat f_{\sigma}}{\partial x_1}= \cdots = \frac{\partial\hat  f_{\sigma}}{\partial x_{n+2}}=0$$ 
has no non-zero solutions over $\overline{\mathbb F}_q$. 

Suppose $\sigma$ has codimension at least two. We claim either $V_{n+1}$ or $V_{n+2}$ is not in $\sigma$. Otherwise, let
$\phi: \mathbb R^{n+2}\to \mathbb R$ be a linear function such that $\phi|_{\Delta_\infty(\hat f)}\leq 1$ and $\phi|_{\Delta_\infty(\hat f)}$ 
attains it maximum $1$ exactly at $\sigma$. Then we have 
$$\phi(V_{n+1})=\phi(V_{n+2})=1,\quad \phi(V_1), \ldots, \phi(V_n),\phi(V_{n+3})\leq 1.$$
Moreover, we have $\phi(V_i)<1$ for at least one $i\in \{1, \ldots, n, n+3\}$ since otherwise $\sigma$ contains all 
$V_j$ $(j=1, \ldots, n+3)$ and thus has codimension $1$.
We have $$V_1+\cdots+V_n+V_{n+3} -m(V_{n+1}+V_{n+2})=0.$$
So 
\begin{eqnarray*}
\phi(V_i)=-\sum_{j \in \{1, \ldots, n, n+3\}-\{i\}} \phi(V_j)+m(\phi(V_{n+1})+\phi(V_{n+2}))\geq -n+2m=1,
\end{eqnarray*}
which contradicts to $\phi(V_i)<1.$ This proves the claim. It implies that $\sigma$ is a face of $\delta_1$ or $\delta_2$ described 
in the proof of (i). For such $\sigma$, 
the proof of (i) already shows that $\hat f_\sigma$ is nondegenerate and thus the above system has no nonzero solutions. 

It remains to check the codimension one face $\sigma = \delta$.  
We want to show that the system 
$$\frac{\partial \hat f}{\partial x_1}= \cdots = \frac{\partial \hat f}{\partial x_{n+2}}=0$$ 
has no non-zero solutions over $\overline{\mathbb F}_q$. This system is 
$$w\frac{x_{n+2}^{m}}{{x_1^2 \cdots x_n}} -1 =\cdots =  w\frac{x_{n+2}^{m}}{{x_1\cdots x_n^2}} -1 = 1 -\frac{x_{n+2}}{x_{n+1}^2} = \frac{1}{x_{n+1}} - mw\frac{x_{n+2}^{m-1}}{x_1 \cdots x_n}=0. $$
It forces $x_1=\cdots =x_n$ and $x_{n+2}=x_{n+1}^2$. Substituting this into 
the last equation and noting that now $2m=n+1$, we obtain 
$$\frac{1}{x_{n+1}} - mw\frac{x_{n+1}^{2m-2}}{x_1^n}= \frac{1}{x_{n+1}} - mw\frac{x_{n+1}^{n-1}}{x_1^n} =0,$$
that is, 
$$\Big(\frac{x_1}{x_{n+1}}\Big)^n = mw.$$
On the other hand, the first equation becomes 
$$0=w\frac{x_{n+1}^{2m}}{{x_1^2 \cdots x_n}} -1 = w\frac{x_{n+1}^{n+1}}{{x_1^{n+1}}} -1 =  \frac{1}{m} \frac{x_{n+1}}{x_1}-1.$$
This implies that $x_1/x_{n+1} = 1/m$. Substituting this into the previous equation, we obtain 
$w = m^{-(n+1)}$. 
This is the only exception for non-degeneracy.  

(iii) Assume $n+1 > 2m$. Let $\delta$ be a codimension $1$ face of $\Delta$ not containing the origin in $\mathbb R^{n+2}$. If $\delta$ does not contain $V_{n+2}$, then $\delta$ must contain $\{V_1, \cdots, V_{n+1}, V_{n+3}\}$. 
The equation of $\delta$ is 
$$\delta: \ x_1+\cdots +x_{n+1} + \frac{n+1}{m} x_{n+2} = 1.$$
Substituting $V_{n+2}$ into this equation, one finds that 
$$0 + \cdots + 0 + (-1) + \frac{n+1}{m} = \frac{n+1 - m}{m} >1,$$
contradicting to our assumption that $\delta$ is a codimension $1$ face of $\Delta$ not containing the origin. 
Thus, $\delta$ must contain $V_{n+2}$. If $\delta$ does not contain $V_{n+1}$, then $\delta$ must contain $\{V_1, \cdots, V_n, V_{n+2}, V_{n+3}\}$. 
The equation of $\delta$ is 
$$\delta: \ x_1+\cdots +x_{n} + \frac{n+1-m}{m}x_{n+1} + \frac{n+1}{m} x_{n+2} = 1.$$
Substituting $V_{n+1}$ into this equation, one finds that 
$$0 + \cdots + 0  + \frac{n+1-m}{m}  + 0= \frac{n+1 - m}{m} > 1.$$ Contradiction again. 
It follows that $\delta$ must contain both $V_{n+1}$ and $V_{n+2}$. The face $\delta$ can only miss one of 
the remaining $n+1$ vertices $\{ V_1, \cdots, V_n, V_{n+3}\}$. This leads to the following $n+1$ 
faces with equations 
\begin{eqnarray*}
\delta_1:   & (2m-n)x_1 +\cdots + x_{n+1} + 2x_{n+2}=1, \\ 
& \hat f_{\delta_1}=-(x_2 +\cdots +x_n)+x_{n+1} + 
\frac{x_{n+2}}{x_{n+1} }-w\frac{x_{n+2}^m}{x_1\cdots x_n},\\
&\vdots&\\
\delta_n: & \ x_1 +\cdots + (2m-n)x_n+x_{n+1} + 2x_{n+2}=1, \\ 
& \  \hat f_{\delta_n}= -(x_1 +\cdots +x_{n-1})+x_{n+1} + 
\frac{x_{n+2}}{x_{n+1} }-w\frac{x_{n+2}^m}{x_1\cdots x_n},\\
\delta_{n+1}: & \ x_1 +\cdots + x_{n+1} + 2x_{n+2}=1, \\
& \  \hat f_{\delta_{n+1}}=-( x_1 +\cdots +x_n)+x_{n+1} + 
\frac{x_{n+2}}{x_{n+1} }.
\end{eqnarray*}
As above, one checks that these $\delta_i$ ($1\leq i \leq n+1$) are indeed codimension $1$ faces of $\Delta$ 
not containing the origin. Furthermore, each $\hat f_{\delta_i}$ is diagonal whose exponent matrix has 
determinant $\pm 1$ and thus non-degenerate. This proves that $\hat f$ 
itself is non-degenerate for all $w\not =0$, and by the facial triangulation, we have
$$(n+2)! \mathrm{vol}(\Delta_{\infty}(\hat f)) = \sum_{i=1}^{n+1}(n+2)! \mathrm{vol}(\Delta_{\infty}(\hat f_{\delta_i})) = n+1.$$
The proof of the lemma is complete. 
\end{proof}

We conclude this paper with several questions for further study. 

\begin{question}
(i)  What can be said for the toric exponential sum at the exceptional singular fibre $w= m^{-(n+1)}$ when $n+1=2m$? 

(ii) Katz's  exotic Kloosterman sheaf is always pure. But the exotic inverted Kloosterman sheaf studied in this paper is not always pure. It would be interesting to give an explicit description of the weight decomposition for the 
exotic inverted Kloosterman sheaf. The answer seems to be more subtle as it would also depend on the multiplicative character $\chi$. For 
sufficiently general $\chi$, the sheaf should be pure. For the case $\chi=1$ and $(n+1, p)=1$, the weight decomposition is 
determined  in \cite{LW}. 

(iii) Zelingher \cite{EZ} recently links Katz's  exotic Kloosterman sheaf to special values of the Bessel function for 
generic irreducible representations of general linear groups over $\mathbb F_q$. It would be interesting to know if 
there is a similar connection for the exotic inverted Kloosterman sheaf. 

\end{question}

\end{document}